%% file: nonlinear_sensitivity.tex
\documentclass[11pt,a4paper]{amsart}

\usepackage{graphicx}      
\usepackage{amsmath}
\usepackage{amssymb}
\usepackage{xcolor}
\usepackage{amsthm}
\usepackage{appendix}
\usepackage{epstopdf}
\usepackage{gensymb}
\usepackage{placeins}
\usepackage{graphicx}

\usepackage{cleveref}
\usepackage{enumerate}
\usepackage{graphicx}
\usepackage{pgfplots,tikz}
\usetikzlibrary{shapes,arrows,arrows.meta,decorations.markings,decorations.text}
\usepackage{mathtools}
\usepackage{algorithmic}
\tikzstyle{arrow} = [draw, -latex']
\usepackage[centering,marginparwidth=2.5cm]{geometry}
\usepackage{float}

\makeatletter
\newcommand\mathcircled[1]{%
	\mathpalette\@mathcircled{#1}%
}
\newcommand\@mathcircled[2]{%
	\tikz[baseline=(math.base)] \node[draw,circle,inner sep=1pt] (math) {$\m@th#1#2$};%
}
\makeatother

\newcommand{\vnorm}[1]{\left\lVert#1\right\rVert}
\newcommand{\norm}[1]{\vnorm{#1}}
\DeclareMathOperator*{\esssup}{ess\,sup}

\newtheorem{thrm}{Theorem}
\newtheorem{crllr}[thrm]{Corollary}
\newtheorem{lmm}[thrm]{Lemma}
\newtheorem{prpstn}[thrm]{Proposition}
\newtheorem{ssmptn}[thrm]{Assumption}
\newtheorem{xmpl}[thrm]{Example}
\theoremstyle{definition}
\newtheorem{dfntn}[thrm]{Definition}
\newtheorem{rmrk}[thrm]{Remark}
\newtheorem{lgrthm}[thrm]{Algorithm}

\crefname{crllr}{Corollary}{Corollaries}
\crefname{thrm}{Theorem}{Theorems}
\crefname{lmm}{Lemma}{Lemmas}
\crefname{xmpl}{Example}{Examples}
\crefname{rmrk}{Remark}{Remarks}
\crefname{dfntn}{Definition}{Definitions}
\crefname{lgrthm}{Algorithm}{Algorithms}
\crefname{prpstn}{Proposition}{Propositions}
\crefname{ssmptn}{Assumption}{Assumption}

\begin{document}
\title[Abstract nonlinear sensitivity and turnpike analysis]{Abstract nonlinear sensitivity and turnpike analysis and an application to semilinear parabolic PDEs}
\author[Lars Grüne, Manuel Schaller, and Anton Schiela]{Lars Grüne$^{1}$, Manuel Schaller$^{1,2}$, and Anton Schiela$^{1}$}
\thanks{}

\thanks{$^{1}$Universität Bayreuth, Institute of Mathematics, Germany}
\thanks{$^{2}$Technische Universit\"at Ilmemau, Institute of Mathematics, Germany (e-mail: manuel.schaller@tu-ilmenau.de).}

	\thanks{{\bf Acknowledgments:}	This work was supported by the German Research Foundation (DFG) under grant numbers GR 1569/17-1 and SCHI 1379/5-1.}

\begin{abstract}
 We analyze the sensitivity of the extremal equations that arise from the first order necessary optimality conditions of nonlinear optimal control problems with respect to perturbations of the dynamics and of the initial data. To this end, we present an abstract implicit function approach with scaled spaces. We will apply this abstract approach to problems governed by semilinear PDEs. In that context, we prove an exponential turnpike result and show that perturbations of the extremal equation's dynamics, e.g., discretization errors decay exponentially in time. The latter can be used for very efficient discretization schemes in a Model Predictive Controller, where only a part of the solution needs to be computed accurately. We showcase the theoretical results by means of two examples with a nonlinear heat equation on a two-dimensional domain. 

\end{abstract}

\maketitle
\section{Introduction}
In this paper we provide an abstract framework for exponential sensitivity analysis of nonlinear optimal control problems with respect to perturbations of the right-hand side of the first-order necessary optimality conditions. We extend previous results regarding linear quadratic optimal control problems, where an exponential damping property was proven for problems governed by non-autonomous parabolic equations in \cite{Gruene2018c} and for problems governed by autonomous general evolution equations in \cite{Gruene2019}. The main tool in these works is a bound on the solution operator of the first order optimality conditions that is independent of the time horizon, which can be deduced under stabilizability and detectability assumptions. In this work considering nonlinear problems, we derive an implicit function theorem that allows for estimates in scaled spaces, where, if the solution operator to the linearized system is bounded independently of the time horizon, all involved neighborhoods and constants are independent of the time horizon. As a consequence, we obtain exponential sensitivity results for nonlinear problems, stating that perturbations of the first order optimality conditions decay exponentially in time.

This sensitivity analysis has several important applications. First, exponential decay of perturbations gives rise to efficient numerical methods. In \cite{Gruene2018c,Gruene2020,Shin2020} efficient discretization methods for Model Predictive Control (MPC) were presented. MPC is a feedback control technique, where an optimal control problem on a infinite or indefinite horizon is approximated by a series of optimal control problems on finite horizons $T>0$. In every iteration of the MPC algorithm, an optimal control for a problem on the time interval $[0,T]$ is computed. Then, an initial part of the control up to time $\tau>0$ is used as a feedback, where often $\tau \ll T$. The resulting state $x(\tau)$ is measured or estimated and the process is repeated on the horizon $[\tau,T+\tau]$ with initial datum $x(\tau)$. For an in-depth introduction, its approximation properties and treatment of various aspects related to MPC the interested reader is referred to \cite{Gruene2016,Gruene2016b,Rawlings2017}. As only an initial part of the computed solution is of relevance for the feedback, an exponential decay of perturbation allows for efficient solution of the corresponding OCPs on space and time grids that are finer on $[0,\tau]$ than on the remainder $[\tau,T]$. This specialized discretization can be done in a a priori fashion \cite{Gruene2018c,Na2020} or with goal oriented a posteriori methods \cite{Gruene2020}. Further, in \cite{Na2020}, a Schwarz decomposition method was considered which also strongly leverages the exponential decay of perturbations.

A second application of the proposed sensitivity analysis is the derivation of turnpike properties. A particular turnpike property is the steady state turnpike property, which is a feature of solutions to autonomous optimal control problems. In a nutshell, it states that the solution of the optimal control problem on a long time horizon resides near a steady-state of the dynamics, the so-called turnpike, for the majority of the time. We briefly recall some of the existing literature on turnpike analysis. The linear quadratic case for control of evolution equations was considered in \cite{Breiten2018,Damm2014,Gruene2018c,Gruene2019,Gugat2016,Gugat2018}. A turnpike property for shape optimization was introduced in \cite{Lance2019,Trelat2017}. Nonlinear finite dimensional problems were considered in \cite{Trelat2015}, including the case of nonlinear initial and terminal conditions. This was extended to a Hilbert space setting in \cite{Porretta2013,Porretta2016,Trelat2016}. A turnpike result for the two-dimensional Navier-Stokes equations was obtained in \cite{Zamorano2018}. These works analyze the turnpike property via the extremal equations and are of local nature, i.e., the initial resp.\ terminal value for state and adjoint need to be close to the turnpike. Further, in \cite{Pighin2020}, a semi-global turnpike result for a semilinear heat equation with initial datum of arbitrary size is given, under the assumption that either the state reference trajectory is small or that the control acts everywhere. A geometric approach to tackle nonlinear problems was presented in \cite{Sakamoto2019}. Another approach that leads to global turnpike properties is stability analysis based on a dissipativity concept. Motivated by the seminal papers by Willems \cite{Willems1972a,Willems1972}, a notion of dissipativity for optimal control problems can be defined, where the supply rate is defined via the cost functional. Assuming this dissipativity property, a global turnpike result for states and controls was deduced in, e.g., \cite{Faulwasser2017,Gruene2016a,Gruene2018a} or \cite[Proposition 8.15]{Gruene2016b}. Under the assumption of a global turnpike property of states and controls, a global turnpike property for the corresponding adjoint states was derived in \cite{Faulwasser2020}. The connection of dissipativity and the turnpike property is also discussed in \cite{Gruene2019a,Trelat18}. Recently, turnpike properties for non-observable systems \cite{Faulwasser2020a,Pighin2020a}, for problems arising in deep learning \cite{Esteve2020} and for fractional parabolic problems \cite{Warma2020} were presented. Further, the connection of turnpike properties and long-time behavior of the Hamilton-Jacobi equation was analyzed in \cite{Esteve2020a}.

In this work, we provide a framework to lift the approach of analyzing the extremal equations' solution operator that was considered in \cite{Gruene2018c,Gruene2019,Warma2020} to the nonlinear case. Thus, our approach and the turnpike results results in this paper are similar to \cite{Trelat2016,Trelat2015} in the sense that we also consider stabilizability conditions to derive local results via analysis of the extremal equations. A particular novelty of this work is the implicit function theorem \Cref{thm:implfunc} that provides a flexible and rigorous tool to conclude a local nonlinear turnpike result by means of analyzing the solution operator of the linearized equations. In that context, turnpike properties in norms particularly tailored to the corresponding regularity properties of the underlying problem can be deduced. The choices could range from, e.g., uniform norms in a general semigroup context, to integral norms with values in Sobolev spaces for parabolic problems. Here, we provide a particular application to semilinear parabolic problems.

Our approach is conceptually simple: the nonlinear problem is considered as a perturbation of a linearized problem and an implicit function theorem is applied. Then exponential sensitivity and turnpike behavior are inherited from the linearized problem to the nonlinear problem under a smallness condition. What makes the analysis delicate is that substantial results of this form require independence of the involved quantities from the length of the time interval $T$ under consideration. For example, the exponential rate of decay $\mu$ should not degenerate as $T\to \infty$. It is thus necessary to check all estimates for uniformity in $T$. 

This work is organized as follows. After introducing the problem of interest in \Cref{sec:nonlin:optcond}, we present an implicit function theorem in \Cref{subsec:nonlin:implfunc} that allows for scaled estimates for solutions of the nonlinear first-order necessary optimality conditions if the linearization of the latter is $T$-uniformly continuous and $T$-uniformly invertible, where $T$ is the time horizon of the optimal control problem. The former assumption will be analyzed in \Cref{subsec:nonlin:superpos}, whereas we will verify the latter for a class of problems involving a semilinear heat equation under stabilizability conditions in \Cref{sec:nonlin:imprreg}. Consequently, we present in \Cref{sec:tp_sensi_semilin} the main results for semilinear problems: We provide a turnpike result in \Cref{thm:nonlin:maxreg:turnpike} and a sensitivity result with respect to perturbations of the dynamics in \Cref{thm:nonlin:maxreg:sensitivity}. Finally, we present numerical examples that illustrate the theoretical findings for semilinear equations and further provide an example with a boundary controlled quasilinear equation that is not covered by theoretical results yet, motivating directions of further research. We will illustrate the turnpike property and propose and evaluate an a priori discretization method specialized for MPC.
\section{Setting and preliminaries}
\label{sec:nonlin:optcond}
We briefly define the nonlinear optimal control problem of interest and formally derive the optimality conditions. Let $\Omega \subset \mathbb{R}^n$, $n\geq 2$ be a bounded domain with smooth boundary. Further, suppose that $(V,\|\cdot\|_V)$ is a separable and reflexive Banach space such that $V\hookrightarrow L_2(\Omega)\cong L_2(\Omega)^* \hookrightarrow V^*$ form a Gelfand triple. We will abbreviate $\langle \cdot,\cdot\rangle := \langle \cdot,\cdot\rangle_{L_2(\Omega)}$ and $\|\cdot\| := \|\cdot\|_{L_2(\Omega)}$. Further we denote
\begin{align*}
	W([0,T]) := \{v\colon [0,T]\to V\,|\,v\in L_2(0,T;V), v'\in L_2(0,T;V^*)\},
\end{align*}
where the time derivative is meant in a weak sense. The control space $U$ will be assumed to be a Hilbert space with scalar product denoted by $\langle \cdot,\cdot\rangle_U$ and induced norm $\|\cdot\|_U$. We consider the following nonlinear optimal control problem.
\begin{align}
	\label{def:OCP}
	\begin{split}
		\min_{(x,u)}\, {J}(x,u):=\int_0^T\bar{J}&(t,x(t))+ \|R(u(t)-u_\text{d}(t))\|_{U}^2\,dt\\
		\text{s.t. }x'(t) &= \bar{A}(x(t)) + \bar{B}u(t) +d(t)\\
		\qquad x(0) &= x_0,
	\end{split}
\end{align}
where $x_0\in H$, $d\in L_2(0;T;V^*)$, $J(x,u)$ is a sufficiently smooth functional, $\bar{B}\colon U\to V^*$ is a continuous and linear operator, and $\bar{A} \colon V \to V^*$ is a sufficiently smooth operator.
We define $B\colon L_2(0,T;U)\to L_2(0,T;V^*)$ via
\begin{align*}
	\langle Bu,\lambda\rangle_{L_2(0,T;V^*) \times L_2(0,T;V)} &:= \int_0^T \langle \bar{B}u(t), \lambda(t)\rangle_{V^* \times V}\,dt
	\intertext{and  $A\colon L_2(0,T;V)\to L_2(0,T;V^*)$ by}
	\langle A(x),\lambda \rangle_{L_2(0,T;V^*)\times L_2(0,T;V)} &:= \int_0^T \langle \bar{A}(x(t)),\lambda(t)\rangle_{V^*\times V}\,dt.
\end{align*} for $\lambda\in L_2(0,T;V)$. We will assume that the optimal control problem \eqref{def:OCP} has a solution $(x,u)\in W([0,T])\times L_2(0,T;U)$. One important ingredient for establishing this property are the classical lower-semi-continuity and coercivity properties of the objective functional. A second factor can be to establish the existence of a continuous control to state map. In the linear case, i.e., if ${A}(x)={A}x$ this follows if $-A$ satisfies a G\aa rding inequality, cf.\ \cite[Theorem 3.4]{Schiela2013}:
\begin{align*}
	\exists \omega \in \mathbb{R}, \alpha > 0 : \quad \alpha \|x\|_{L_2(0,T;V)}^2 \le -\langle Ax,x\rangle_{L_2(0,T;V^*) \times L_2(0,T;V)}+\omega \|x\|_{L_2(0,T;L_2(\Omega))}^2.
\end{align*}
For solvability of semilinear equations with globally Lipschitz semilinearities we refer to \cite[Chapter 6]{Pazy83} and \cite[Chapter 5]{Troeltzsch2010}. Locally Lipschitz semilinearities were treated in \cite{Raymond1999}, where global existence of solutions was ensured by sufficiently regular data $(Bu+d,x_0)$, such that the solution is bounded, i.e., $x\in L_\infty(0,T;L_\infty(\Omega))$. For a in-depth analysis of optimal control problems governed by quasilinear parabolic equations, the interested reader is referred to \cite{Bonifacius2018,Casas1995,Ladyzhenskaia1968,Meinlschmidt2017,Papageorgiou1992}.

We introduce a Lagrange multiplier $(\lambda,\lambda_0)\in L_2(0,T;V)\times L_2(\Omega)$ and define the Lagrange function via 
\begin{align*}
	L(x,u,(\lambda,\lambda_0)):= J(x,u) &+ \langle x'-Ax,\lambda\rangle_{(L_2(0,T;V) \times L_2(\Omega))^* \times (L_2(0,T;V) \times L_2(\Omega))}\\& -  \langle Bu+d,\lambda\rangle_{L_2(0,T;V^*) \times L_2(0,T;V)} + \langle x(0)-x_0,\lambda_0\rangle.
\end{align*}
Proceeding formally we obtain the first-order necessary optimality conditions
\begin{align}
	\label{eq:nonlin:exactextremal}
	L'(x,u,\lambda)=
	\begin{pmatrix}
		J_x(x,u)-  \lambda' - A'(x)^*\lambda\\
		\lambda(T)\\
		R^*Ru- B^*\lambda\\
		x'-A(x) - Bu-d\\
		x(0)-x_0
	\end{pmatrix}
	=0.
\end{align}
with $\lambda \in W([0,T])$ as $J_x(x,u)\in L_2(0,T;V^*)$, cf.\ \cite[Proposition 3.8]{Schiela2013}. Setting $Q:=R^*R$ and $L_r(x,\lambda):=L(x,Q^{-1}B^*\lambda + u_\text{d},\lambda)$, the reduced extremal equations read
\begin{align}
	\label{eq:nonlin:exactextremal_red}
	L_r'(x,\lambda)&=
	\begin{pmatrix}
		J_x(x,u)-  \lambda' - A'(x)^*\lambda\\
		\lambda(T)\\
		x'-A(x) - BQ^{-1}B^*\lambda -Bu_\text{d}-d\\
		x(0)-x_0
	\end{pmatrix}
	=0.
\end{align}
We present two perturbations of the extremal equations \eqref{eq:nonlin:exactextremal_red} that we aim to analyze in this work.
In order to obtain a sensitivity result with respect to perturbations of the dynamics, we introduce a perturbation $\varepsilon=(\varepsilon_1,\varepsilon_T,\varepsilon_2,\varepsilon_0)\in (L_2(0,T;V^*)\times L_2(\Omega))^2$, which could result from, e.g., temporal or spatial discretization errors. We denote by $(\tilde{x},\tilde{u},\tilde{\lambda})\in W([0,T])\times L_2(0,T;U)\times W([0,T])$ the solution of this perturbed system, i.e.,
\begin{align}
	\label{eq:nonlin:perturbedextremal}
	L_r'(\tilde{x},\tilde{\lambda})=
	\begin{pmatrix}
		J_x(\tilde{x},\tilde{u})-  \tilde{\lambda}' - A'(\tilde{x})^*\tilde{\lambda}\\
		\tilde{\lambda}(T)\\
		\tilde{x}'-A(\tilde{x}) - BQ^{-1}B^*\tilde{\lambda}-Bu_\text{d}-d\\
		\tilde{x}(0)-x_0
	\end{pmatrix}
	=\begin{pmatrix}
		\varepsilon_1\\
		\varepsilon_T\\ 
		\varepsilon_2\\
		\varepsilon_0
	\end{pmatrix}
\end{align}
and $\tilde u = Q^{-1}B^*\tilde{\lambda} +u_\text{d}$.
Further, to derive a turnpike result, i.e., a sensitivity result with respect to perturbations of the initial and terminal condition, we will consider the steady-state problem as a perturbation of the first order conditions of the dynamic problem. In that context, we will always assume that the cost functional is given by $J(x,u)=\int_0^T \bar{J}(x(t))+\frac12\|R(u(t)-u_d)\|^2\,dt$, i.e., $\bar{J}$ itself does not explicitly depend on time and $u_d\in U$, $R\in L(U,U)$, $d\in V^*$ are independent of time. To indicate this time-independence denote $\bar{R}:= R$, $\bar{u}_d:= u_d$ and $\bar{d}:=d$. The corresponding steady-state problem then reads
\begin{align}
	\label{def:SOCP}
	\begin{split}
		\min_{(\bar{x},\bar{u})} \,\bar{J}&(\bar{x})+\frac12\|\bar{R}(\bar{u}-\bar{u}_d)\|_U^2\\
		\text{s.t. } -\bar{A}(\bar{x}) &= \bar{B}\bar{u} + \bar{d},
	\end{split}
\end{align}
where we again assume that there is a minimizer $(\bar{x},\bar{u}) \in V\times U$. For $(\bar{x},\bar{u},\bar{\lambda}) \in V\times U\times V$, we define the Lagrange function of the steady-state system $\bar{L}(\bar{x},\bar{u},\bar{\lambda}):= \bar{J}(\bar{x},\bar{u})-\langle \bar{A}(\bar{x})-\bar{B}\bar{u},\lambda\rangle_{V^*\times V}$, which leads to the first order conditions
\begin{align}
	\label{eq:soptcond}
	\bar{L}'(\bar{x},\bar{u},\bar{\lambda})=
	\begin{pmatrix}
		\bar{J}_x(\bar{x})-\bar{A}'(\bar{x})^*\bar{\lambda}\\
		\bar{R}^*\bar{R}\bar{u}- \bar{B}^*\bar{\lambda}\\
		-\bar{A}(\bar{x}) - \bar{B}\bar{u}-\bar{d}\end{pmatrix}
	=0.
\end{align}
Eliminating the control via $\bar{u}=\bar{Q}^{-1}\bar{B}^*\bar{\lambda} +u_\text{d}$, where $\bar{Q}=\bar{R}^*\bar{R}$ and defining the reduced static Lagrangian $\bar{L}_r(\bar{x},\bar{\lambda}):=\bar{L}(\bar{x},\bar{Q}^{-1}\bar{B}^*\bar{\lambda} +\bar{u}_\text{d},\bar{\lambda})$, this steady-state system can be written as a perturbation of the dynamic extremal equations by interpreting $\bar{\lambda}$ and $\bar{x}$ as functions constant in time and by adding $\bar{\lambda}'=\bar{x}'=0$ and initial resp.\ terminal value to the equations, i.e.,
\begin{align}
	\label{eq:nonlin:staticextremal}
	L_r'(\bar{x},\bar{\lambda})=
	\begin{pmatrix}
		J_x(\bar{x},\bar{u})-  \bar{\lambda}' - A'(\bar{x})^*\bar{\lambda}\\
		\bar{\lambda}(T)\\
		\bar{x}'-A(\bar{x}) - BQ^{-1}B^*\bar{\lambda}-Bu_\text{d}-\bar{d}\\
		\bar{x}(0)-x_0
	\end{pmatrix}
	=\begin{pmatrix}
		0\\
		\bar{\lambda}\\
		0\\
		\bar{x}-x_0
	\end{pmatrix}.
\end{align}

To obtain localized estimates in time, we consider a smooth scaling function $s\colon \mathbb{R}^{\geq 0}\to\mathbb{R}$ with $s(t) > 0$ for all $t\in \mathbb{R}^{\geq 0}$ and endow $L_p(0,T;V)$ with the scaled norm 
\begin{align}
	\label{eq:nonlin:scalednorms}
	\|x\|_{L^s_p(0,T;V)}:=\|sx\|_{L_p(0,T;V)}
\end{align}
for any $1\leq p\leq \infty$.
The equivalence of this norm to the standard $L_p(0,T;V)$-norm follows from the positivity of $s$ as we get for $1\leq p < \infty$ that
\begin{align}
	\label{eq:nonlin:normsequi1}
	\min_{t\in [0,T]}s(t)\left(\int_0^T\!\! \|x(t)\|_V^p\,dt\right)^{\frac{1}{p}}\!\!\leq\left(\int_0^T\!\! \|s(t)x(t)\|_V^p\,dt\right)^{\frac{1}{p}}\!\! \leq \max_{t\in [0,T]}s(t)\left(\int_0^T \!\!\|x(t)\|_V^p\,dt\right)^{\frac{1}{p}}
\end{align}
and 
\begin{align}
	\label{eq:nonlin:normsequi2}
	\min_{t\in [0,T]}s(t)\esssup_{t\in [0,T]}\|x(t)\|_V\leq\esssup_{t\in [0,T]}\|s(t)x(t)\|_V \leq \max_{t\in [0,T]}s(t)\esssup_{t\in [0,T]}\|x(t)\|_V.
\end{align}
As $L_p(0,T;V)$ with the standard norm is a Banach space, by the equivalence of the norms above, $\left(L_p(0;T;V),\|\cdot\|_{L_p^s(0,T;V)}\right)$ is also a Banach space. Note that the equivalence of norms can deteriorate for $T\to \infty$ depending on the scaling function. Later in \Cref{sec:nonlin:imprreg}, we will consider a closed operator $A:D(A)\subset L_2(\Omega)\to L_2(\Omega)$ that is a generator of an analytic semigroup in $L_2(\Omega)$, where $D(A)$ is the domain of $A$ endowed by the graph norm $\|\cdot\| + \|A \cdot\|$. We will impose either homogeneous Neumann or homogeneous Dirichlet boundary conditions and thus the domain will, in our case,  be either $D(A) = \{v\in H^2(\Omega)\,|\, \frac{\partial}{\partial\nu_{A}} v=0\}$, where $\frac{\partial}{\partial\nu_{A}}$ is the conormal derivative corresponding to $A$, or $D(A) = H^2(\Omega)\cap H^1_0(\Omega)$, i.e., $V=D((-A)^\frac{1}{2})$.  Correspondingly, we will set $V=H^1(\Omega)$ or $V=H^1_0(\Omega)$ depending on the choice of boundary conditions. We will denote
\begin{align*}
	W^{1,2}(0,T,D(A),L_2(\Omega)) &:= \{v \in L_2(0,T;D(A))\,|\, v'\in L_2(0,T;L_2(\Omega))  \},\\
	\|v\|_{W^{1,2}(0,T;D(A),L_2(\Omega))} &:= \|v\|_{L_2(0,T;D(A))} + \|v'\|_{L_2(0,T;L_2(\Omega))},
\end{align*}
where the time derivative is meant in a weak sense. We have the $T$-independent embedding $W^{1,2}(0,T,D(A),L_2(\Omega))\hookrightarrow C(0,T;V)$, cf.\ \cite[Part II-1, Remark 4.1, Remark 4.2]{Bensoussan2007}. For this vector-valued Sobolev space, we will also utilize a scaled norm, i.e.,
\begin{align*}
	\|v\|_{W^{1,2}_s(0,T,D(A),L_2(\Omega))} := \|sv\|_{W^{1,2}(0,T,D(A),L_2(\Omega))}.
\end{align*}
For the scaling terms we have in mind, i.e., exponential functions, one can straightforwardly show that the norm $\|v\|_{W^{1,2}_s(0,T,D(A),L_2(\Omega))}=\|sv\|_{L_2(0,T;D(A))}+\|(sv)'\|_{L_2(0,T;L_2(\Omega))}$ is equivalent to $\|sv\|_{L_2(0,T;D(A))}+\|sv'\|_{L_2(0,T;L_2(\Omega))}=\|v\|_{L_2^s(0,T;D(A))}+\|v'\|_{L_2^s(0,T;L_2(\Omega))}$. Thus, by the equivalence of scaled and unscaled $L_2$-norms shown above, $\|\cdot\|_{W^{1,2}_s(0,T,D(A),L_2(\Omega))}$ is equivalent to the standard norm $\|\cdot\|_{W^{1,2}(0,T,D(A),L_2(\Omega))}$ with constants strongly depending on $T$. Hence, $\left(W^{1,2}(0,T,D(A),L_2(\Omega)),\|\cdot\|_{W^{1,2}_s(0,T,D(A),L_2(\Omega))}\right)$ is a Banach space. Finally, whenever we write $V^{s(t)}$ for either $V=H^1(\Omega)$ or $V=H^1_0(\Omega)$ and $t\in [0,T]$, we mean $V$ endowed with the equivalent norm $s(t)\|\cdot\|_{V}$. 
If $X$ is a Banach space we denote by $L(X):=L(X,X)$ the space of bounded linear operators from $X$ to $X$.
\section{An abstract framework for sensitivity analysis}
\label{subsec:nonlin:implfunc}
Considering the nonlinear equation \eqref{eq:nonlin:exactextremal_red} and the perturbations \eqref{eq:nonlin:perturbedextremal} resp.\ \eqref{eq:nonlin:staticextremal}, the question we aim to answer is the following: How do $\tilde{z}=(\tilde{x},\tilde{u})$ and $\bar{z}=(\bar{x},\bar{u})$ differ from $z=(x,u)$ depending on $(\varepsilon_1,\varepsilon_T,\varepsilon_2,\varepsilon_0)$ and $(0,\bar{\lambda},0,\bar{x}-x_0)$, respectively? In particular, we aim to obtain results localized in time by means of the scaled norms introduced above.
To this end, denoting by $Z$ a solution space and by $E$ a perturbation space, we introduce a nonlinear operator
\begin{align*}
	G\colon Z \times E \to E
\end{align*}
defined by
\begin{align}
	\label{defn:nonlin:G}
	G(z,\varepsilon) := L_r'(z)-\varepsilon \qquad \forall (z,\varepsilon)\in Z\times E.
\end{align}
It is clear that
\begin{itemize}
	\item $G(z,0) = 0$ for any solution $z=(x,\lambda)$ of the dynamic problem \eqref{eq:nonlin:exactextremal_red},
	\item $G(\bar{z},(0,\bar{\lambda},0,\bar{x}-x_0))=0$ for any solution $\bar{z}=(\bar{x},\bar{\lambda})$ of the static problem \eqref{eq:soptcond},\item $G(\tilde{z},(\varepsilon_1,\varepsilon_T,\varepsilon_2,\varepsilon_0))=0$ for any solution $\tilde{z}=(\tilde{x},\tilde{\lambda})$ of the perturbed dynamic problem \eqref{eq:nonlin:perturbedextremal}.
\end{itemize} 
In \Cref{sec:nonlin:imprreg} we apply the abstract theory of this chapter to a class of semilinear parabolic problems. In that context, we will utilize the smoothing effect of parabolic equations and we will obtain estimates in 
\begin{align*}
	Z_s&=\left(L_{p}(0,T;L_{p}(\Omega))\cap W^{1,2}(0,T;D(A),L_2(\Omega)),\|\cdot\|_{L_{p}^s(0,T;L_{p}(\Omega))\cap W_s^{1,2}(0,T;D(A),L_2(\Omega))}\right)^2 ,\\
	E_s&=\left(L_2(0;T;L_2(\Omega)),\|\cdot\|_{L_2^s(0,T;L_2(\Omega))}\right) \times V^{s(T)}\\&\qquad \qquad\qquad\qquad\qquad\qquad \qquad\times  \left(L_2(0;T;L_2(\Omega)),\|\cdot\|_{L_2^s(0,T;L_2(\Omega))}\right)\times V^{s(0)}
\end{align*} to derive sensitivity and turnpike results. The perturbations of the dynamics are assumed to belong to an $L_2$-space, whereas the perturbations of the initial values have to belong to $V$. This regularity of the data leads to solutions with values a.e.\ in $D(A)$ that have a weak time derivative with values a.e.\ in $L_2(\Omega)$ by maximal parabolic regularity, cf.\ \cite[Part II-1, Section 3]{Bensoussan2007}. In order to obtain exponential sensitivity estimates, we equip both the solution space and the space of right-hand sides by a scaled norm.
\subsection{An implicit function theorem}
We now present an implicit function theorem that allows for estimates in scaled norms in a very general setting. A particular feature of the following implicit function theorem is the tracking of dependencies of the neighborhoods of perturbations and solutions in scaled and unscaled norms. This allows us to formulate a criterion that renders these neighborhoods independent of $T$, namely $T$-uniform continuity and $T$-uniform invertibility of the operator corresponding to the linearized first-order necessary conditions. This uniformity in $T$ is crucial to derive meaningful turnpike and sensitivity results. The assumption of $T$-independence of the solution operators norm is also a central assumption in the linear quadratic setting and in that case can be achieved under stabilizability and detectability assumptions, cf.\ \cite[Corollary 3.16]{Gruene2018c} and \cite[Theorem 10]{Gruene2019}. We will derive a similar property for the linearized system in \Cref{sec:nonlin:imprreg}. Finally we emphasize that even though the scaled and unscaled norms are equivalent, the involved constants in case of exponential scalings strongly depend on $T$. Thus, this equivalence of norms can not be directly used to derive estimates, motivating a refined analysis as carried out in the following theorem.
\begin{thrm}
	\label[thrm]{thm:implfunc}
	Let $(Z,\|\cdot\|_Z)$ and $(E,\|\cdot\|_E)$  be Banach spaces, let $\|v\|_{Z_s}$ resp.\ $\|v\|_{E_s}$ be equivalent norms on $Z$ resp.\ $E$ and set $Z_s:=\left(Z,\|\cdot\|_{Z_s}\right)$ and $E_s:=(E,\|\cdot\|_{E_s})$.
	Consider the mapping $G\colon Z\times E \to E$ defined in \eqref{defn:nonlin:G} with $G(z^0,\varepsilon^0)=0$ for $(z^0,\varepsilon^0)\in Z\times E$. Assume the following:
	\begin{enumerate}[i)]
		\item $G_z(z^0,\varepsilon^0)$ is continuously invertible in $L(Z,E)$.
		\item It holds that \begin{align*}
			\delta_\varepsilon(z^1,z^2):=	\frac{\|G(z^1,\varepsilon)-G(z^2,\varepsilon) - G_z(z^0,\varepsilon^0)(z^1-z^2)\|_{E}}{\|z^1-z^2\|_Z} \to 0,
		\end{align*}
		if $z^1,z^2\to z^0$ in $Z$ and $ \varepsilon \to \varepsilon^0$ in $E$.
		\item It holds that \begin{align*}
			\delta^s_\varepsilon(z^1,z^0):=	\frac{\|G(z^1, \varepsilon)-G(z^0,\varepsilon) - G_z(z^0,\varepsilon^0)(z^1-z^0)\|_{E_s}}{\|z^1-z^0\|_{Z_s}} \to 0,
		\end{align*}
		if $z^1\to z^0$ in $Z$ and $ \varepsilon \to \varepsilon^0$ in $E$.
	\end{enumerate}	
	Then there is $r_E,r_Z \geq 0$, such that for every $\varepsilon$ satisfying $\|\varepsilon-\varepsilon^0\|_{E}\leq r_E$ there exists $z^*(\varepsilon)\in Z$ such that $\|z^*(\varepsilon)-z^0\|_Z\leq r_Z$ and $G(z^*,\varepsilon)=0$. Further, we have the estimate
	\begin{align}
		\label{eq:nonlin:scaledconst}
		\|z^*(\varepsilon)-z^0\|_{Z_s}\leq c\|\varepsilon-\varepsilon^0\|_{E_s}.
	\end{align}
	Moreover we have the following $T$-uniformity:
	\begin{itemize}
		\item If the convergence of ii) is uniform in $T$ and $\|G_z(z^0,\varepsilon^0)^{-1}\|_{L(E,Z)}$ is bounded independently of $T$, then $r_Z$ and $r_E$ can be chosen independently of $T$.
		\item  If, additionally, the convergence of iii) is uniform in $T$ and $\|G_z(z^0,\varepsilon^0)^{-1}\|_{L(E_s,Z_s)}$ is bounded independently of $T$, then the constant in \eqref{eq:nonlin:scaledconst} is independent of $T$. 
	\end{itemize}
\end{thrm}
\begin{proof}  Throughout this proof, we denote $B_{r_Z}^Z(z^0):=\{z\in Z\,|\,\|z-z^0\|_Z\leq r_Z\}$ and analogously $B_{r_E}^E(\varepsilon^0):=\{\varepsilon\in E\,|\,\|\varepsilon-\varepsilon^0\|_E\leq r_E\}$.
	For $k\in \mathbb{N}^0$ let $\delta z^{k} := -G_z(z^0,\varepsilon^0)^{-1}G(z^{k},\varepsilon)$ and $z^{k+1}=z^k + \delta z^{k}$. 
	As 
	\begin{align*}
		\delta z^{k+1} = -G_z(z^0,\varepsilon^0)^{-1}\left(G(z^{k+1},\varepsilon)-G(z^k,\varepsilon) - G_z(z^0,\varepsilon^0)(z^{k+1}-z^k)\right)
	\end{align*}
	we have with ii) that
	\begin{align}
		\label{eq:nonlin:impl0}
		\|\delta z^{k+1}\|_Z=\|G_z(z^0,\varepsilon^0)^{-1}\|_{L(E,Z)}\delta_\varepsilon(z^{k+1},z^k)\|\delta z^k\|_Z
	\end{align}
	where $\delta_\varepsilon(z^{k+1},z^k)\to 0$, if $z^{k+1},z^k\to z^0$ in $Z$ and $\varepsilon\to \varepsilon^0$ in $E$. We now choose a neighborhood $B_{r_Z}^Z(z^0)\times B_{r_E}^E(\varepsilon^0)$ such that $\delta_\varepsilon(z_1,z_2)\leq \frac{1}{2\|G_z(z^0,\varepsilon^0)^{-1}\|_{L(E,Z)}}$ for all $\varepsilon \in B_{r_E}^E(\varepsilon^0)$ and $z_1,z_2 \in B_{r_Z}^Z(z^0)$. Further, by continuity of $G(z^0,\varepsilon)$ in $\varepsilon$, continuous invertibility of $G(z^0,\varepsilon^0)$ and as $G(z^0,\varepsilon^0)=0$, we can further decrease $r_E$ such that
	\begin{align*}
		\|\delta z^0\|_Z = \|G_z(z^0,\varepsilon^0)^{-1}G(z^0,\varepsilon)\|_Z \leq \frac{r_Z}{2}
	\end{align*} 
	for $\varepsilon\in B_{r_E}^E(\varepsilon^0)$.
	Thus, we get 
	\begin{align}
		\label{eq:nonlin:impl1}
		\|\delta z^{k+1}\|_Z\leq \|G_z(z^0,\varepsilon^0)^{-1}\|_{L(E,Z)}\delta_\varepsilon(z^{k+1},z^k)\|\delta z^k\|_Z \leq \left(\frac{1}{2}\right)^k\|\delta z^0\|_Z.
	\end{align}
	Hence, 
	\begin{align}
		\label{eq:nonlin:impl2}
		\|z^{k+1}-z^0\|_Z\leq \sum_{i=0}^{k}\|\delta z^{k}\|_Z \leq \frac{1}{1-\frac{1}{2}}\|\delta z^0\|_Z \leq  r_Z
	\end{align}
	and hence inductively, $z^k\in B_{r_Z}^Z(z^0)$ for all $k\in \mathbb{N}$. Thus, by completeness of $Z$, the iteration $z^k = z^0+\sum_{i=0}^k\delta z^k$ converges to an element $z^*\in B_{r_Z}^Z(z^0)$ and as $\delta_\varepsilon(z^*,z^k)\leq \frac{1}{2\|G_z(z^0,\varepsilon^0)^{-1}\|_{L(E,Z)}}$ we get
	\begin{align*}
		\|G_z(z^0,\varepsilon^0)^{-1}\left(G(z^*,\varepsilon)-G(z^k,\varepsilon)-G_z(z^0,\varepsilon^0)(z^*-z^k)\right)\|_Z\leq \frac{1}{2}\|z^*-z^k\|_Z.
	\end{align*}
	Hence, by the reverse triangle inequality, we get
	\begin{align*}
		\|G_z(z^0,\varepsilon^0)^{-1}G(z^*,\varepsilon)\|_Z \leq \|\delta z^{k}+z^*-z^k\|_Z + \frac{1}{2}\|z^*-z^k\|_Z \to 0
	\end{align*} for $k\to \infty$ and thus $G(z^*,\varepsilon)=0$.
	To obtain an estimate in the scaled norms, we compute
	\begin{align*}
		\|z^*-z^0\|_{Z_s}\leq \|z^*-z^1\|_{Z_s} + \|\delta z^0\|_{Z_s}.
	\end{align*}
	We further estimate with $z^1 = z^0+\delta z^0$, $\delta z^0 = -G_z(z^0,\varepsilon^0)^{-1}G(z^0,\varepsilon)$, by $iii)$ and $z^* \in B_{r_Z}^Z(z^0)$ after possibly further decreasing $r_Z$ and $r_E$ such that $\delta ^s_\varepsilon(z^*,z^0)\leq \frac{1}{2\|G_z(z^0,\varepsilon^0)^{-1}\|_{L(E_s,Z_s)}}$ that
	\begin{align*}
		\|z^*-z^1\|_{Z_s} &= \|G_z(z^0,\varepsilon^0)^{-1}(G_z(z^0,\varepsilon^0)(z^*-z^1))\|_{Z_s} \\
		&= \|G_z(z^0,\varepsilon^0)^{-1}(G_z(z^0,\varepsilon^0)(z^*-z^0) - (G(z^*,\varepsilon)-G(z^0,\varepsilon)))\|_{Z_s}\\
		&\leq \frac12 \|z^*-z^0\|_{Z_s}.
	\end{align*}
	Hence by the particular structure of $G$, i.e., $G(z^0,\varepsilon)=G(z^0,\varepsilon^0)+\varepsilon-\varepsilon^0=\varepsilon-\varepsilon^0$ we obtain
	\begin{align*}
		\frac12\|z^*-z^0\|_{Z_s} \leq  \|\delta z^0\|_{Z_s}=  \|G_z(z^0,\varepsilon^0)^{-1}G(z^0,\varepsilon)\|_{Z_s}\leq \|G_z(z^0,\varepsilon^0)^{-1}\|_{L(E_s, Z_s)}\|\varepsilon-\varepsilon^0\|_{E_s}
	\end{align*}
	which concludes the proof.
\end{proof}
We have two particular applications of \cref{thm:implfunc} in mind. First, to derive a turnpike result, we set $z^0=(\bar{x},\bar{\lambda})$ solving the static extremal equations \eqref{eq:nonlin:staticextremal}, $\varepsilon^0 = (0,\bar{\lambda},0,\bar{x}-x_0)$, and $\varepsilon=0$ to derive an estimate on the difference of $(x,\lambda)$ and $(\bar{x},\bar{\lambda})$ in scaled norms with scaling function $s(t)=\frac{1}{e^{-\mu t}+e^{-\mu (T-t)}}$. Second, in order to obtain a sensitivity result, we set $z^0=(x,\lambda)$ solving the exact dynamic extremal equations \eqref{eq:nonlin:exactextremal_red}, $\varepsilon^0=(0,0,0,0)$, and $\varepsilon=(\varepsilon_1,\varepsilon_T,\varepsilon_2,\varepsilon_0)$ to derive an estimate on the difference of $(x,\lambda)$ and $(\tilde{x},\tilde{\lambda})$ solving the perturbed extremal equations \eqref{eq:nonlin:perturbedextremal} in scaled norms with scaling function $s(t)=e^{-\mu t}$.
\begin{rmrk}
	Due to its generality, \cref{thm:implfunc} can also be applied to general evolution equations, i.e., hyperbolic equations or alternatively to elliptic problems, where the scaling could act in space. For the latter one can prove an exponential decay property of the influence of right-hand sides in space, a well-known property for elliptic equations, without knowledge of the Greens function.
\end{rmrk}
A crucial point in the proof of the implicit function theorem, i.e., \cref{thm:implfunc}, is to ensure that the series generated by $G_z(z^0,\varepsilon^0)^{-1}G(z^k,\varepsilon)$ converges in $Z$. In the assumptions of the theorem, this is ensured by i) and ii), i.e., differentiability of the nonlinear operator and continuous invertibility of the linearization. As we will see in the following section, in general, the image of a nonlinear map, e.g., $G(z^k,\varepsilon)$ has lower integrability than its argument $z^k$. Thus, we need a smoothing effect of the solution operator to the linearized problem, e.g., $G_z(z^0,\varepsilon^0)^{-1}$ to make up for this loss of regularity. We rigorously prove this property for parabolic problems in \Cref{sec:nonlin:imprreg}.
\subsection{Superposition operators and $T$-uniform continuity}
\label{subsec:nonlin:superpos}
In order to rigorously verify assumptions ii)-iii) in \cref{thm:implfunc}, we employ the concept of superposition operators. We will only consider continuity and differentiability of these operators in $L_p$-spaces and the reader is referred to \cite[Section 4.3.3]{Troeltzsch2010} for a short introduction and \cite{Appell1990,Goldberg1992} for an in-depth treatment of these topics in Sobolev and Lebesgue spaces of abstract functions. Intuitively, a superposition operator is a nonlinear map between function spaces defined via an, e.g., scalar nonlinear function by superposition. The following definition of a superposition operator is adapted from \cite[Section 4.3.1]{Troeltzsch2010} and \cite[Section 2]{Goldberg1992}. 
\begin{dfntn}[Superposition operator]
	\label[dfntn]{defn:nonlin:superpos}
	Let $W_1$ and ${W_2}$ be real valued Banach spaces. Consider a mapping $f\colon W_1\to {W_2}$. Then the mapping $\Phi$ defined by
	\begin{align*}
		\Phi(x)(s) = f(x(s)) \qquad \text{for } s\in S
	\end{align*}
	assigns to an (abstract) function $x\colon S\to W_1$ a new (abstract) function $z\colon S\to {W_2}$ via the relation $z(s) = f(x(s))$ for $s\in S$ and is called an \textit{(abstract) Nemytskij} operator or \textit{(abstract) superposition operator}.
\end{dfntn}
The image of a superposition operator in $L_p$-spaces can be characterized under growth and boundedness conditions.
\begin{prpstn}
	\label[prpstn]{prop:nonlin:mapsto}
	Let $W_1$ and ${W_2}$ be real valued Banach spaces.	Let $f\colon {W_1}\to {{W_2}}$ be continuous. For $1\leq p,q <\infty$ let 
	\begin{align}
		\label{eq:nonlin:growth}
		\|f(w)\|_{W_2} \leq c_1+c_2\|w\|_{W_1}^{\frac{p}{q}} \qquad \forall\,w\in {W_1}
	\end{align}
	for constants $c_1\in \mathbb{R}$ and $c_2 \geq 0$. Then the corresponding superposition operator maps $L_p(S;{W_1})$ into $L_q(S;{W_2})$.
\end{prpstn}
\begin{proof}
	See \cite[Theorem 1]{Goldberg1992}.
\end{proof}
\label{subsec:nonlin:lpcont}
The following proposition shows that if a superposition operator maps one $L_p$-space into another, continuity can be derived immediately.
\begin{prpstn}[Continuity of superposition operators]
	\label[prpstn]{prop:nonlin:cont}
	Let $W_1$ and ${W_2}$ be real valued Banach spaces. Let $f\colon {W_1}\to {W_2}$ be continuous and $1\leq p\leq \infty$, $1\leq q <\infty$. If the induced superposition operator $\Phi$ maps $L_p(S;{W_1})$ into $L_q(S,{W_2})$, then it is continuous. If $f$ is locally Lipschitz, then it is continuous as a map from $L_\infty(S;W_1)$ to $L_\infty(S;W_2)$.
\end{prpstn}
\begin{proof}
	For the first part, see \cite[Theorem 4]{Goldberg1992}. For the case with $p=q=\infty$, see \cite[Theorem 1, Theorem 5]{Goldberg1992} or \cite[Lemma 4.11]{Troeltzsch2010}.
\end{proof}
We note that in the case $p=q=\infty$, a growth bound of the form \eqref{prop:nonlin:cont} is not needed. In particular, the property that the superposition $\Phi$ operator maps $L_\infty(S;W_1)$ to $L_\infty(S;W_2)$ can be concluded by local Lipschitz continuity of the underlying function $f$, cf.\ \cite[Theorem 1, Theorem 5]{Goldberg1992} or \cite[Lemma 4.11]{Troeltzsch2010}.

Differentiability of superposition operators plays a key role in applying the implicit function theorem. The following result obtained in \cite[Theorem 7]{Goldberg1992} gives sufficient conditions for Fr\'echet differentiability.
\begin{prpstn}[Differentiability of superposition operators]
	\label[prpstn]{prop:nonlin:diff}
	Let $1\leq q <p<\infty$. Assume that $f\colon {W_1}\to {W_2}$ is continuously Fr\'echet differentiable. Moreover, let the superposition operator defined by 
	\begin{align*}
		\Psi(x)(s)=f'(x(s))\quad \text{for }s\in S
	\end{align*}
	be continuous from $L_p(S;{W_1})$ to $L_r(S;L({W_1},{W_2}))$ with $r=\frac{pq}{p-q}$. Then the superposition operator $\Phi$ induced by $f$ is continuously Fr\'echet differentiable and the Fr\'echet derivative 
	\begin{align*}
		\Phi'\colon L_p(S;{W_1})\to L(L_p(S;{W_1}),L_q(S;{W_2}))
	\end{align*}
	is given by $\Psi$, i.e.,
	\begin{align*}
		(\Phi'(x)\delta x)(s)=\Psi(x)(s)\delta x(s) \quad \text{for }s\in S,\,\delta x \in L_p(S;{W_1}).
	\end{align*}
\end{prpstn}The conditions given in \cref{prop:nonlin:diff} are also necessary in the following sense: If a superposition operator is differentiable from $L_p(\Omega)$ to $L_q(\Omega)$ with $1\leq p=q<\infty$, then it is affine-linear. If it is differentiable from $L_p(\Omega)$ to $L_q(\Omega)$ with $1\leq p<q\leq \infty$, then it has to be constant, cf.\ the discussion in \cite[Section 3.1]{Goldberg1992} and \cite[Theorem 3.12]{Appell1990}.

We briefly illustrate this concept by means of a particular example of a polynomial nonlinearity.
\begin{xmpl}
	\label[xmpl]{rem:nonlin:polys}
	Consider $W_1={W_2}=\mathbb{R}$ and $S=\Omega\subset\mathbb{R}^n$ bounded with $n\in \mathbb{N}$. Then the nonlinear function $f \colon \mathbb{R}\to\mathbb{R}$, $f(w)=w^d$, $ d\in \mathbb{N}$, defines a superposition operator $\underline{\Phi}$ via the relation
	\begin{align*}
		\underline\Phi(x)(\omega) = f(x(\omega))= x(\omega)^d \quad \text{for }\omega\in \Omega
	\end{align*}
	for $x\colon \Omega\to \mathbb{R}$. With \cref{prop:nonlin:mapsto} we have for all $1\leq p_1<\infty$ that $\underline\Phi \colon L_{dp_1}(\Omega)\to L_{p_1}(\Omega)$ and by \cref{prop:nonlin:cont}, this mapping is continuous. Further, as $f'(w) = dw^{d-1}$, we can define a continuous superposition operator $\underline{\Psi}\colon L_{(d-1)p_2}(\Omega)\to L_{p_2}(\Omega)$ for all $1\leq p_2 <\infty$ corresponding to the derivative $f'(w)$. Consequently we can set $(d-1)r=p_1$  in \cref{prop:nonlin:diff}, i.e., $p_2= \frac{p_1}{d}$ which yields $\underline{\Phi}$ to be continuously Fr\'echet differentiable as a mapping from $L_{dp_1}(\Omega)$ to $L_{p_1}(\Omega)$ and $\underline{\Phi}'=\Psi$ in the sense of \cref{prop:nonlin:diff}.
	Consider now $1\leq p_3<\infty$, $T>0$, and the nonlinear function $\underline{\Phi}\colon L_{dp_3}(\Omega)\to L_{p_3}(\Omega)$ defined above. Setting $W_1=L_{dp_3}(\Omega)$, ${W_2}=L_{p_3}(\Omega)$, and $S=[0,T]$ in \cref{defn:nonlin:superpos}, we define a second superposition operator $\Phi$ for $x\colon [0,T]\to L_{dp}(\Omega)$ via the relation
	\begin{align*}
		\Phi(x)(t) = \underline{\Phi}(x(t)) \quad \text{for }t\in [0,T].
	\end{align*}
	Proceeding analogously as before, we obtain that $\Phi$ is continuous and differentiable as operator from $L_{dp_3}(0,T;L_{dp_1}(\Omega))$ to $L_{p_3}(0,T;L_{p_1}(\Omega))$ for $1\leq p_1,p_3 < \infty$. We thus obtained from a scalar nonlinear function a nonlinear mapping from one space of abstract functions into another one by applying \cref{defn:nonlin:superpos} twice.
\end{xmpl}
In order to render the radii $r_Z$ and $r_E$ and the estimate \eqref{eq:nonlin:scaledconst} independent of $T$, we have to discuss the $T$-dependence of continuity moduli of superposition operators in unscaled and scaled $L_p$-spaces as introduced at the end of \Cref{sec:nonlin:optcond} with norms defined in \eqref{eq:nonlin:scalednorms}.  
\begin{dfntn}[$T$-uniform continuity]
	Let $W_1,W_2$ be real-valued Banach spaces. We say that an operator \mbox{$\Psi \colon L_p(0,T;W_1)\to L_q(0,T;W_2)$} is $T$-uniformly continuous if for all $x^0\in L_p(0,T;W_1)$ and for all $\varepsilon>0$ there is $\delta > 0$ independent of $T$ such that if $\|\delta x\|_{L_p(0,T;W_1)}<\delta$ then
	\begin{align*}
		\|\Psi(x^0+\delta x)-\Psi(x^0)\|_{ L_q(0,T;W_2)} <\varepsilon.
	\end{align*}
\end{dfntn}
\begin{lmm}
	If the constants $c_1$ and $c_2$ of the growth condition \eqref{eq:nonlin:growth} are independent of $T$, then the continuity of the induced superposition operator it $T$-uniform.
\end{lmm}
\begin{proof}
	The proof follows directly by the fact that the references establishing continuity under growth conditions do not assume the domain $S$ to be bounded, cf.\ \cite[Chapter 3]{Appell1990} and \cite{Goldberg1992}. 
\end{proof}
\begin{xmpl}[\cref{rem:nonlin:polys} revisited]
	\label[xmpl]{ex:revisited}
	We briefly illustrate the previous lemma at the example $f(w)=w^d$. In that case it is clear that the growth condition \eqref{eq:nonlin:growth}, i.e., 
	\begin{align*}
		|f(w)| \leq c_1+c_2|w|^{\frac{p}{q}}
	\end{align*}
	holds with $p=3q$, $c_1=0$ and $c_2=1$, i.e., $f$ induces a $T$-uniformly continuous superposition operator from $L_{dq}(0,T;L_{dp}(\Omega))$ to $L_{q}(0,T;L_p(\Omega))$. As $[0,T]$ and $\Omega$ are bounded, one can show that continuity also holds from $L_{\hat{d}q}(0,T;L_{\hat{d}p}(\Omega))$ to $L_{q}(0,T;L_p(\Omega))$ for $\hat{d}>d$, however, with constants that depend on $T$ and $|\Omega|$. This means, that the functional analytic framework has to be chosen particularly suited to the nonlinearity to render the constants and hence the continuity uniform in $T$.
\end{xmpl}
The following lemma shows that if a superposition operator has a $T$-uniformly continuous Fr\'echet derivative, the convergence in ii) and iii) of \cref{thm:implfunc} can be shown to be $T$-uniform.
\begin{lmm}
	\label[lmm]{lem:nonlin:scaleddiff}
	Let $W_1$ and $W_2$ be Banach spaces, $1\leq p\leq \infty$ and let $\Phi\colon L_p(0,T;W_1) \to L_q(0,T;W_2)$ have a $T$-uniformly continuous Fr\'echet derivative $\Phi'$. Then,
	\begin{align*}
		\frac{\|\Phi(x^1)-\Phi(x^2)-\Phi'(x^0)(x^1-x^2)\|_{{L_q(0,T;W_2)}}}{\|x^1-x^2\|_{L_p(0,T;W_1)}} \to 0
	\end{align*}
	uniformly in $T$ if $x^1,x^2 \to x^0$ in $L_p(0,T;W_1)$. Moreover,
	\begin{align*}
		\frac{\|\Phi(x^0+\delta x)-\Phi(x^0)-\Phi'(x^0)\delta x\|_{{L^s_q(0,T;W_2)}}}{\|\delta x\|_{L^s_p(0,T;W_1)}} \to 0
	\end{align*}
	uniformly in $T$ if $\delta x \to 0$ in $L_p(0,T;W_1)$.
\end{lmm}
\begin{proof}
	We compute with the fundamental theorem of calculus, cf.\ \cite[p.51]{Hinze09}, that
	\begin{align*}
		\|\Phi(x^1)&-\Phi(x^2)-\Phi'(x^0)(x^1-x^2)\|_{{L_q(0,T;W_2)}}\\
		&= \|\int_0^1 \Phi'(x^2+\theta(x^1-x^2)) -  \Phi'(x^0) \,d\theta (x^1-x^2)\|_{{L_q(0,T;W_2)}}\\
		&\leq\bigg(\sup_{\theta \in [0,1]}\|\Phi'(x^2+\theta(x^1-x^2))-\Phi'(x^2)\|_{L({{L_p(0,T;W_1)},L_q(0,T;W_2)})}\\&\qquad\qquad\qquad+\|(\Phi'(x^0)-\Phi'(x^2))\|_{{L(L_p(0,T;W_1),L_q(0,T;W_2))}}\bigg)\|x^1-x^2\|_{L_p(0,T;W_1)}.
	\end{align*}
	The first claim follows by $T$-uniform continuity of $\Phi'$.
	For the second claim in scaled norms with scaling function $s$, we compute analogously
	\begin{align*}
		\|\Phi(x^0+\delta x)-\Phi(x^0)&-\Phi'(x^0)\delta x\|_{{L^s_q(0,T;W_2)}}\\&=\|s\left(\Phi(x^0+\delta x)-\Phi(x^0)-\Phi'(x^0)\delta x\right)\|_{{L_q(0,T;W_2)}}\\&=\sup_{\theta \in [0,1]}\|\Phi'(x^0)-\Phi'(x^0+\theta\delta x)\|_{L({{L_p(0,T;W_1)},L_q(0,T;W_2)})}\|s\delta x\|_{L_p(0,T;W_1)}\\&\leq\sup_{\theta \in [0,1]}\|\Phi'(x^0)-\Phi'(x^0+\theta\delta x)\|_{L({L_p(0,T;W_1)},{L_q(0,T;W_2)})}\|\delta x\|_{{L^s_p(0,T;W_1)}},
	\end{align*}
	which concludes the proof.
\end{proof}
Hence, it turns out that whenever the superposition operators are differentiable with $T$-uniformly continuous derivative, the uniform convergence needed in \cref{thm:implfunc} ii) and iii) to obtain $T$-uniform neighborhoods holds true. The last thing to prove to apply the implicit function theorem is the $T$-uniform estimate on the solution operator to the linearized first-order optimality system, i.e., $G_z(x^0,\varepsilon^0)^{-1}$ in unscaled and scaled spaces. In the following we will derive such a bound for a wide class of semilinear problems that provides flexibility in the norms to match the functional analytic framework where one established $T$-uniform continuity.
\section{A $T$-independent bound for the extremal equations' solution operator for semilinear parabolic problems}
\label{sec:nonlin:imprreg}
In this part we will rigorously verify the assumptions of the abstract implicit function theorem, i.e., \cref{thm:implfunc}, for a class of semilinear heat equations. The analysis in this part is heavily motivated by the approach taken in \cite{Raymond1999}, where the authors derive a Maximum Principle for optimal control problems governed by semilinear parabolic PDEs. In that work it is shown that for sufficiently smooth data, the solution $x$ of a semilinear parabolic PDE with monotone nonlinearity indeed satisfies $x\in L_\infty(0,T;L_\infty(\Omega))$. This allows for existence results globally in time without global Lipschitz conditions on the nonlinearity. For convenience of the reader, we briefly introduce the setting considered in \cite{Raymond1999}. To this end, we assume that the PDE of interest is semilinear parabolic, i.e., $\bar{A}(x)=\mathcal{A}x-f(x)$ with $f'(x)\geq c_0$ for $c_0\in \mathbb{R}$ and that $B\in L(L_2(\Omega_c)),L_2(\Omega))$ for a control domain $\Omega_c\subset \Omega$ which includes, e.g., the case of distributed control. The operator $-\mathcal{A}$ is considered to be an elliptic differential operator of second order, i.e.,
\begin{align}
	\label{def:nonlin:ell}
	\mathcal{A}x:= \sum_{i,j=0}^nD_i(a_{ij}D_j x),
\end{align}
where $k\geq 0$, $a_{ij}=a_{ji}\in C(\bar{\Omega},\mathbb{R})$ and $a_{i,j}(\omega)v\cdot v> 0$ for all $\omega \in \Omega$ and $v\in \mathbb{R}^n$. By $\frac{\partial x}{\partial \nu_{\mathcal{A}}}(t,s)=\sum_{i,j=0}^n a_{ij}(s)D_jx(t,s)\nu_i(s)$ we denote the conormal derivative of $x$, where $\nu=(\nu_1,\dots,\nu_n)$ is the outward unit normal to $\partial \Omega$. 
We consider the domain
\begin{align}
	\label{eq:nonlin:dom}
	D(\mathcal{A})= \{v\in C^2(\Omega)\,|\,v= 0\text{ on }\partial \Omega\} \qquad \text{or}\qquad D(\mathcal{A})= \{v\in C^2(\Omega)\,|\,\frac{\partial v}{\partial \nu_{\mathcal{A}}}= 0\text{ on }\partial \Omega\}
\end{align}
for either homogeneous Dirichlet or homogeneous Neumann boundary conditions. We assume w.l.o.g. that there is $\alpha>0$ such that
\begin{align}
	\label{eq:nonlin:ellip}
	-\int_{\Omega} \mathcal{A}vv\,d\omega \geq \frac{\alpha}{2}\|v\|^2_{H^1(\Omega)}
\end{align}
for $v\in D(\mathcal{A})$. In case of Dirichlet boundary conditions this immediately follows with integration by parts and the Poincar\'e inequality. For Neumann boundary conditions, we can replace $\mathcal{A}x$ by $(\mathcal{A}-kI)$ for any $k>0$ by $\bar{A}(x)=\mathcal{A}x-f(x)=(\mathcal{A}-kI)x+kx-f(x)$ and redefine $f(x):=f(x)-kx$. 

It can be shown that for all $1\leq l < \infty$ the closure $A_l$ of $\mathcal{A}$ in $L_l(\Omega)$ generates an analytic semigroup in $L_l(\Omega)$ and we abbreviate $A=A_2$. For $1<l<\infty$, the domain is given by $D(A_l)=\{v\in W^{2,l}(\Omega)\,|\,v=0\text{ on }\partial \Omega\}$ or $D(A_l)=\{v\in W^{2,l}(\Omega)\,|\,\frac{\partial v}{\partial\nu_{\mathcal{A}}}=0\text{ on }\partial \Omega\}$, depending on the choice in \eqref{eq:nonlin:dom}. Additionally, the spectrum of $A_l$ does not depend on $1\leq l<\infty$. For details we refer to \cite{Rothe1982} and \cite[Section 3]{Raymond1999}.

Correspondingly, depending on the choice of boundary conditions above, we will set $V=H^1(\Omega)$ in the case of homogeneous Neumann boundary conditions or $V=H^1_0(\Omega)$ in the case of homogeneous Dirichlet boundary conditions.

\begin{ssmptn}
	\label[ssmptn]{as:nonlin:Linfty}
	Assume that $(x,\lambda)\in L_\infty(0,T;L_\infty(\Omega))$ for any solution $(x,u,\lambda)$ of \eqref{eq:nonlin:exactextremal} and $(\bar{x},\bar{\lambda})\in L_\infty(\Omega)$ for any solution $(\bar{x},\bar{u},\bar{\lambda})$ of \eqref{eq:soptcond}.
\end{ssmptn}
As stated before, in order to render this assumption satisfied, one usually assumes that the data of \eqref{def:OCP} and \eqref{def:SOCP} are sufficiently smooth and that the nonlinearity is monotone. Under these assumptions, boundedness of solutions in time and space for parabolic problems was proven in \cite{Raymond1999}. Similarly, for semilinear elliptic equations, a proof for continuity of solutions can be found in \cite{Casas1993}. The interested reader is also referred to the respective parts in the monograph \cite{Troeltzsch2010}.
Under \cref{as:nonlin:Linfty} we can conclude by smoothness of $J(x,u)$, $\bar{J}(x)$, and $f(x)$ with \cref{prop:nonlin:cont} that the corresponding superposition operators are continuous from $L_\infty(\Omega)$ to $L_\infty(\Omega)$ and from $L_\infty(0,T;L_\infty(\Omega))$ to $L_\infty(0,T;L_\infty(\Omega))$, respectively.
We further introduce a square root property for the second derivative of the reduced Lagrange function with respect to the state.
\begin{lmm}
	\label[lmm]{ass:ssc}
	Let $(x^0,\lambda^0)\in L_\infty(0,T;L_\infty(\Omega))^2$ and $(L_r)_{xx}(x,\lambda)=J_{xx}(x^0)+f''(x^0)^*\lambda^0 \in\\ L_\infty(0,T;L_\infty(\Omega))$ induce a nonnegative multiplication operator, i.e., for $v\colon [0,T]\times \Omega \to \mathbb{R}$
	\begin{align*}
		((L_r)_{xx}(x^0,\lambda^0) v) (t,\omega):= (L_r)_{xx}(x^0,\lambda^0)(t,\omega)\cdot v(t,\omega)
	\end{align*}
	and $(L_r)_{xx}(x^0,\lambda^0)(t,\omega) \geq 0$ for a.e.\ $t\in [0,T]$ and $\omega\in \Omega$.
	Then, there is a multiplication operator $C\in L(L_{p_1}(0,T;L_{p_2}(\Omega)))$ for all $1\leq p_1,p_2\leq \infty$ defined by 
	\begin{align}
		\label{eq:CstarC}
		(Cv)(t,\omega) := \sqrt{(L_r)_{xx}(x^0,\lambda^0)(t,\omega)}\cdot v(t,\omega)
	\end{align} such that
	\begin{align*}
		(L_r)_{xx}(x^0,\lambda^0) = C^2.
	\end{align*}
\end{lmm}
\begin{rmrk}
	The assumption of $({L}_r)_{xx}(x^0,\lambda^0)$ inducing a multiplication operator is satisfied if the cost functional is of the form 
	\begin{align*}
		{J}(x,u) = \frac{1}{2}\int_0^T\|{x}(t)-x_\text{d}(t)\|^2_{L_2(\Omega_o)}+ \frac{\alpha}{2}\|u(t)\|_U^2\,dt
	\end{align*}
	for $\Omega_o\subset \Omega$ and if the nonlinearity is given by $f(x)=x^3$. In that case, \begin{align*}
		{L}_{xx}(x^0,\lambda^0) = \chi_{\Omega_o} + 6x^0\lambda^0,
	\end{align*}
	where $\chi_{\Omega_o}$ is the characteristic function of the observation region $\Omega_o$. The positivity assumption is fulfilled if, e.g., $\Omega_o=\Omega$ and if $\lambda^0$ and $x^0$ are small in $L_\infty(0,T;L_\infty(\Omega))$, which, for $(x^0,\lambda^0)=(\bar{x},\bar{\lambda})$ or $(x^0,\lambda^0)=(\tilde{x},\tilde{\lambda})$ can be verified by imposing smallness conditions on the data of the underlying steady-state or dynamic OCP, cf.\ \cref{ex:nonlin}. The assumption that $L_{xx}$ is positive semidefinite, was also mad in \cite[Theorem 1]{Trelat2015} and \cite[Theorem 1]{Trelat2016}. As seen in this example and as stated in \cite[Remark 6]{Trelat2015} this assumption is not standard. In particular, it is not clear how to verify it by, e.g., second order sufficient conditions. However, this assumption is crucial to define a square root in the sense of \cref{eq:CstarC}, which itself is necessary to obtain stability results for the linearized system, cf.\ the proof of \cref{cor:nonlin:maxreg:opnorm}.
\end{rmrk}
By means of \cref{as:nonlin:Linfty}, we have that $L_{xx}(x,\lambda)\in L(L_{p_1}(0,T;L_{p_2}(\Omega)))$ for $x,\lambda \in L_{\infty}(0,T,L_\infty(\Omega))$ such that $C\in L(L_{p_1}(0,T;L_{p_2}(\Omega)))$ as defined in \eqref{eq:CstarC} for all $1\leq p_1,p_2 \leq \infty$. 

In order to apply \cref{thm:implfunc}, we will show a bound on the inverse of 
\begin{align*}
	L_r''(x^0,\lambda^0)\colon (L_{p_1}(0,T;L_{p_2}(\Omega))\cap W^{1,2}(0,T,D(A),L_2(\Omega))^2 \to (L_2(0,T;L_{2}(\Omega))\times V)^2,
\end{align*}
where $2\leq p_1,p_2\leq \infty$, $\frac{n}{2}(\frac{1}{2}-\frac{1}{p_2})<\frac{1}{p_1}+\frac{1}{2}$ and $(x^0,\lambda^0)$ either solves the static system \eqref{eq:nonlin:staticextremal} or the dynamic system \eqref{eq:nonlin:exactextremal_red}. We aim to choose $p_1,p_2$ as large as possible to render a wide range of nonlinearities continuous and differentiable with $T$-uniformly continuous derivative in these spaces.

To derive an operator norm we consider the linear system
\begin{align}
	\label{eq:nonlinsys}
	\underbrace{\begin{pmatrix}
			J_{xx}(x^0)+f''(x^0)\lambda^0& -\frac{d}{dt}-\mathcal{A}^*+f'(x^0)\\
			0&E_T\\
			\frac{d}{dt}-\mathcal{A}+f'(x^0) & -BQ^{-1}B^*\\
			E_0&0
	\end{pmatrix}}_{L_r''(x^0,\lambda^0)}
	\begin{pmatrix}
		\delta x\\ \delta \lambda
	\end{pmatrix}
	=
	\begin{pmatrix}
		l_1\\
		\delta \lambda_T\\
		l_2\\
		\delta x_0
	\end{pmatrix}
\end{align}
for $(l_1,\delta \lambda_T,l_2,\delta x_0)\in (L_2(0,T;L_2(\Omega)\times V)$. Note that due to $x^0\in L_\infty(0,T;L_\infty(\Omega))$ and due to the smoothness of $f$, the terms $J_x(x^0)$, $f'(x^0)$, and $f''(x^0)$ are in $L_\infty(0,T;L_\infty(\Omega))$ because of \cref{prop:nonlin:cont} and hence can be interpreted as pointwise multiplications. With slight abuse of notation, we denote by the same symbol the corresponding superposition operator.
We now aim to estimate $(\delta x,\delta \lambda)$ by means of the right-hand side in appropriate norms. To this end, we make the following stabilizability assumptions.
\begin{ssmptn}
	\label[ssmptn]{as:nonlin:stab}
	Let $c_0\in \mathbb{R}$ be such that $c_0\leq f'(w)$ for all $w\in \mathbb{R}$ and let $\bar{C}\in L(L_p(\Omega),L_p(\Omega))$ for all $2\leq p \leq\infty$ be such that $\|\bar{C}v\|_{L_2(0,T;L_2(\Omega))}\leq \|Cv\|_{L_2(0,T;L_2(\Omega))}$ for all $v \in L_2(0,T;L_2(\Omega))$, where $C$ is defined in \eqref{eq:CstarC}. Additionally assume:
	\begin{enumerate}[i)]
		\item $(\mathcal{A}-c_0I,\bar{B})$ is exponentially stabilizable in the sense that for all $1\leq p \leq \infty$ there is $\bar{K}_{\bar{B}}\in L(L_p(\Omega),L_p(\Omega_c))$ satisfying $\bar{B}\bar{K}_{\bar{B}}\in L(L_p(\Omega))$ such that $\mathcal{A}-c_0I+\bar{B}\bar{K}_{\bar{B}}$ satisfies \eqref{eq:nonlin:ellip}.
		\item $(\mathcal{A}-c_0I,\bar{C})$ is exponentially stabilizable in in the sense that for all $1\leq p \leq \infty$ there is $\bar{K}_{\bar{C}}\in L(L_p(\Omega))$ such that $\mathcal{A}-c_0I+\bar{C}^*\bar{K}_{\bar{C}}$ satisfies \eqref{eq:nonlin:ellip}.
		\item $(\mathcal{A}-f'(x^0),C)$ and $(\mathcal{A}-f'(x^0),{B})$ are exponentially stabilizable for $x^0=x$ with $(x,\lambda)$ solving \eqref{eq:nonlin:exactextremal_red} and $x^0=\bar{x}$ for $(\bar{x},\bar{\lambda})$ solving \eqref{eq:nonlin:staticextremal} in the following sense:
		There are operators $K_B\in L(L_2(0,T;L_2(\Omega)),L_2(0,T;L_2(\Omega_c)))$ and $K_C\in L(L_2(0,T;L_2(\Omega)))$ such that 
		\begin{align*}
			-\int_0^T \int_{\Omega} \left(\mathcal{A}-f'(x^0)+ BK_B\right)vv\,d\omega\,dt &\geq \alpha \|v\|_{L_2(0,T;V)}^2\\
			-\int_0^T \int_{\Omega} \left(\mathcal{A}-f'(x^0)+ C^*K_C\right)vv\,d\omega\,dt &\geq \alpha \|v\|_{L_2(0,T;V)}^2
		\end{align*}
		for all $v\in L_2(0,T;D(\mathcal{A}))$.
		\item $\|x^0\|_{L_\infty(0,T;L_\infty(\Omega))}$ and $\|C\|_{L(L_p(0,T;L_p(\Omega)),L_p(0,T;L_p(\Omega)))}$ are bounded independently of $T$ for all $1\leq p\leq \infty$.
	\end{enumerate}
\end{ssmptn}
We briefly comment on these assumptions.
\begin{rmrk}
	\label[rmrk]{rem:maxreg}
	The first two assumptions ensure that the linearized system is stabilizable and detectable and that the closed-loop operators generate a strongly continuous exponentially stable analytic semigroup in $L_p(\Omega)$ for all $1\leq p< \infty$. The third assumption allows us to deduce the $W([0,T])$-bound analogously to \cite[Corollary 3.16]{Gruene2018c}. This stabilizability assumption was introduced for non-autonomous parabolic problems in \cite[Definition 3.6]{Gruene2018c}. The last assumption ensures that the coefficients in the linearized systems are bounded independently of $T$. This is trivially fulfilled for a steady state linearization point $(x^0,\lambda^0)$. In case that the linearization point is the time-dependent optimal solution, this estimate is satisfied if, e.g., a turnpike property in this uniform norm holds. The latter was proven in cf.\ \cite[Theorem 0.2]{Pighin2020} under a smallness assumption on the reference state in case of a tracking type cost functional.
\end{rmrk}
In the following we denote by $\mathcal{A}_\text{cl}$ an operator of the form \eqref{def:nonlin:ell} satisfying \eqref{eq:nonlin:ellip} endowed with a domain defined in \eqref{eq:nonlin:dom}. By $(\mathcal{T}^l_\text{cl}(t))_{t\geq 0}$ we will denote the analytic semigroup on $L_l(\Omega)$ that is generated by the closure of $\mathcal{A}_\text{cl}$ in $L_l(\Omega)$ for $1\leq l<\infty$.
\begin{prpstn}
	\label[prpstn]{prop:extensionestimate}
	There is $\mu_0>0$ and a constant $c>0$ independent of $t$, such that
	\begin{align*}
		\|\mathcal{T}^l_\text{cl}(t) \psi_0\|_{L_q(\Omega)} \leq c\frac{e^{-\mu_0 t}}{t^{\frac{n}{2}(\frac{1}{l}-\frac{1}{q}) }}\|\psi_0\|_{L_l(\Omega)} \qquad \forall t>0
	\end{align*}
	for all $\psi_0 \in L_l(\Omega)$ and $1 \leq l\leq q\leq \infty$ with $l<\infty$. 
\end{prpstn}
\begin{proof}
	See \cite[Lemma 1]{Rothe1982} or \cite[Proposition 12.5]{Amann1983}.
\end{proof}
This stability result for analytic semigroups turns out to be crucial to derive estimates in $L_p$-spaces for large $p$ for, e.g., right-hand sides in $L_2(0,T;L_2(\Omega))$ as performed in the following theorem. As a consequence, we can allow for a wide range of different functional analytic settings, i.e., different choices of integrability parameters. This flexibility can then be leveraged when verifying $T$-uniformity in the context of the superposition operator, i.e., rendering the constants in \cref{prop:nonlin:mapsto} independent of $T$, cf.\ \cref{ex:revisited}. We will again pick up this issue in \cref{rem:nonlin:nonlincost}.

A central tool in the following will be a convolution estimate, similar to the proof of \cite[Lemma 8]{Gruene2019}. The proof is motivated by the approach of \cite[Proposition 3.1]{Raymond1999}.
\begin{thrm}
	\label[thrm]{thm:nonlin:almostthere}
	Let \cref{as:nonlin:stab} hold and let $(\delta x,\delta \lambda)\in W([0,T])^2$ solve \eqref{eq:nonlinsys}. Then, for all $p_1,p_2\geq 2$ satisfying $\frac{n}{2}(\frac{1}{2}-\frac{1}{p_2})<\frac{1}{p_1}+\frac{1}{2}$ with $p_2 < \frac{2n}{n-2}$, there is a constant $c>0$ independent of $T$, such that
	\begin{align*}
		\|(\delta x,\delta \lambda)\|_{W^{1,2}(0,T,D(A),L_2(\Omega))^2}&+\|(\delta x,\delta \lambda)\|_{L_{p_1}(0,T;L_{p_2}(\Omega))^2}\\
		\leq c\big(\|C\delta x\|_{L_2(0,T;L_2(\Omega))} &+ \|B^*\delta\lambda\|_{L_2(0,T;L_2(\Omega_c))}+\|r\|_{(L_2(0,T;L_2(\Omega))\times V)^2}\big),
	\end{align*}
	where $r:=(l_1,\delta \lambda_T,l_2,\delta x_0)$.
\end{thrm}
\begin{proof}
	We will first show the $W^{1,2}(0,T,D(A),L_2(\Omega))$-estimate. To this end, we consider the state equation of \cref{eq:nonlinsys}, i.e.,
	\begin{align*}
		\delta x' +(-\mathcal{A}+f'(x^0))\delta x  - BQ^{-1}B^* \delta \lambda = l_2
	\end{align*}
	with initial condition $\delta x(0)=\delta x_0$. Adding the stabilizing feedback $C^*K_C$ from \cref{as:nonlin:stab} iii), we obtain 
	\begin{align*}
		\delta x' + (-\mathcal{A}+f'(x^0) - C^*K_C)\delta x  = BQ^{-1}B^* \delta \lambda + l_2 -  C^*K_C\delta x.
	\end{align*}
	and testing the equation with $\delta x$, using the coercivity of \cref{as:nonlin:stab} iii) we get
	\begin{align}
		\label{eq:l2est}
		\|\delta x\|_{L_2(0,T;V)} \leq c\left(\|\delta x_0\|_{V} + \|B^*\delta \lambda\|_{L_2(0,T;L_2(\Omega))} + \|l_2\|_{L_2(0,T;L_2(\Omega))} + \|C\delta x\|_{L_2(0,T;L_2(\Omega))}\right).
	\end{align}
	As $\mathcal{A}$ generates an exponentially stable analytic semigroup in $L_2(\Omega)$ by applying the maximal regularity result \cite[Part II-1, Theorem 3.1]{Bensoussan2007} to
	\begin{align*}
		\delta x' - \mathcal{A}\delta x =- f'(x^0)\delta x +BQ^{-1}B^* \delta \lambda + l_2
	\end{align*}
	we obtain
	\begin{align*}
		\|\delta x'\|_{L_2(0,T;L_2(\Omega))}&+\|\mathcal{A}\delta x\|_{L_2(0,T;L_2(\Omega))} \\&\leq c\left(\|\delta x\|_{L_2(0,T;L_2(\Omega))} + \|B^*\delta \lambda\|_{L_2(0,T;L_2(\Omega))} + \|l_2\|_{L_2(0,T;L_2(\Omega))}    \right).
	\end{align*}
	Together with \eqref{eq:l2est} we conclude
	\begin{align*}
		||\delta x&\|_{W^{1,2}(0,T,D(A),L_2(\Omega))}\\&\leq c\left(\|l_2\|_{L_2(0,T;L_2(\Omega))} + \|B^*\delta \lambda\|_{L_2(0,T;L_2(\Omega))} + \|C\delta y\|_{L_2(0,T;L_2(\Omega))}+\|\delta x_0\|_{V}\right).
	\end{align*}
	Proceeding analogously for the adjoint yields the first part of the estimate.
	To obtain the estimate in $L_{p_1}(0,T;L_{p_2}(\Omega))$, we proceed similarly to \cite[Proof of Proposition 3.1]{Raymond1999}.
	Let $\psi_0\in L_l(\Omega)$, $1\leq l<\infty$ and $\psi$ solve the auxiliary problem
	\begin{align*}
		\psi' =(\mathcal{A}-c_0I+\bar{C}^*\bar{K}_{\bar{C}})\psi,  \qquad \psi(0)=\psi_0,
	\end{align*}
	where $\bar{K}_{\bar{C}}$ is a stabilizing feedback for $(\mathcal{A}-c_0I,\bar{C})$ in the sense of \cref{as:nonlin:stab} ii). Thus, by \cref{prop:extensionestimate} for all $1 \leq l\leq q\leq \infty$ with $l<\infty$ and $t> \tau \geq0$ we have the estimate
	\begin{align}
		\label{eq:nonlin:zest}
		\|\psi(t-\tau)\|_{L_q(\Omega)} \leq c \frac{e^{-\mu_0 (t-\tau)}}{(t-\tau)^{\frac{n}{2}(\frac{1}{l}-\frac{1}{q})}}\|\psi_0\|_{L_l(\Omega)}.
	\end{align}
	We compute
	\begin{align}
		\label{eq:nonlinest1}
		\int_{\Omega}\psi_0(\omega)\delta x(t,\omega)\,d\omega= \underbrace{\int_0^t\left(\frac{d}{d\tau}\int_{\Omega}\psi(t-\tau,\omega)\delta x(\tau,\omega)\,d\omega\right)\,d\tau}_{\mathcircled{I}} + \underbrace{\int_{\Omega} \psi(t,\omega)\delta x_0(\omega)\,d\omega}_{\mathcircled{II}}.
	\end{align}
	For the first part of \eqref{eq:nonlinest1} we obtain with $c_0-f'(x^0)\leq 0$ and as $\mathcal{A}$ is self-adjoint that
	\begin{align*}
		\mathcircled{I}&=\int_0^t\left(\int_{\Omega}-\psi'(t-\tau,\omega)\delta x(\tau,\omega)+\psi(t-\tau,\omega)\delta x'(\tau,\omega)\,d\omega\right)d\tau\\
		&=\int_0^t\bigg(\int_\Omega -\mathcal{A}\psi(t-\tau,\omega) \delta x(\tau,\omega)-\bar{C}^*K\psi(t-\tau,\omega)\delta x(\tau,\omega)+c_0\psi(t-\tau,\omega)\delta x(\tau,\omega)
		\\& \qquad +\psi(t-\tau,\omega)\mathcal{A}\delta x(\tau,\omega)-f'(x^0)\psi(t-\tau,\omega)\delta x(\tau,\omega)+l_2(t,\omega)\psi(t-\tau,\omega)\\&\qquad+\bar{B}Q^{-1}\bar{B}^* \delta \lambda(\tau,\omega)\psi(t-\tau,\omega)\,d\omega\bigg)d\tau
		\\&\leq \int_0^t\bigg(\int_\Omega \!-\!\bar{C}^*K\psi(t-\tau,\omega)\delta x(\tau,\omega)\!+\!\psi(t-\tau,\omega)l_2(t,\omega)\!+\!\psi(t-\tau,\omega)\bar{B}Q^{-1}\bar{B}^*\delta \lambda(\tau,\omega)\,d\omega\bigg)d\tau.
	\end{align*}
	In the following we denote by $p_2'$ the dual exponent to $p_2$, i.e., $\frac{1}{p_2}+\frac{1}{p_2'}=1$. Using the exponential stability estimate of \eqref{eq:nonlin:zest} and setting $l=p_2'$ and $q=2$, we obtain for the first summand of \eqref{eq:nonlinest1} that
	\begin{align*}
		\mathcircled{I} \leq c\|\psi_0\|_{L_{p_2'}(\Omega)}\int_0^t \frac{e^{-\mu_0 (t-\tau)}}{(t-\tau)^{\frac{n}{2}(\frac{1}{ p_{2}'}-\frac{1}{2}+\delta)}}\left(\|\bar{C}\delta x(\tau)\|_{L_{2}(\Omega)} + \|\bar{B}^*\delta\lambda(\tau)\|_{L_2(\Omega_c)} + \|l_2(\tau)\|_{L_{2}(\Omega)}\right)\,d\tau .
	\end{align*}
	For the second part of \eqref{eq:nonlinest1} we use H\"older's inequality and \eqref{eq:nonlin:zest} with $q=l=p_2'$ and obtain that
	\begin{align}
		\label{eq:nonlin:criticaldelta}
		\mathcircled{II} \leq \|\psi(t)\|_{L_{p_2'}(\Omega)}\|\delta x_0\|_{L_{p_2}(\Omega)} \leq ce^{-\mu_0 t}|\psi_0\|_{L_{p_2'}(\Omega)}\|\delta x_0\|_{L_{p_2}(\Omega)}.
	\end{align}
	Taking the supremum over all $\psi_0 \in L_{\scriptstyle p_2'}(\Omega)$ yields for any $t\in [0,T]$ that
	\begin{align*}
		\|\delta x(t)\|_{L_{p_2}(\Omega)} &\leq c\int_0^t \frac{e^{-\mu_0 (t-\tau)}}{(t-\tau)^{\frac{n}{2}(\frac{1}{p_{2}'}-\frac{1}{2})}}\left(\|\bar{C}\delta x(\tau)\|_{L_{2}(\Omega)} + \|\bar{B}^*\delta \lambda(\tau)\|_{L_{2}(\Omega_c)} + \|l_2(\tau)\|_{L_{2}(\Omega)}\right)\,d\tau \\&\qquad \qquad \qquad +ce^{-\mu_0 t}\|x_0\|_{L_{p_{2}}(\Omega)}.
	\end{align*}
	We now integrate this inequality over time. To this end, we recall Young's convolution inequality, cf.\ \cite[Theorem II.4.4]{Werner2011}, which states for \mbox{$\frac{1}{p_1}+1=\frac{1}{2}+\frac{1}{h}$} that
	\begin{align*}
		\|w*g\|_{L_{p_1}(\mathbb{R})}\leq \|w\|_{L_{h}(\mathbb{R})}\|g\|_{L_{2}(\mathbb{R})}.
	\end{align*}
	We apply this convolution inequality to 
	\begin{align*}
		g(\tau)&:=\|\bar{C}\delta x(\tau)\|_{L_{2}(\Omega)}+\|\bar{B}^*\delta \lambda(\tau)\|_{L_{2}(\Omega_c)}+\|l_2(\tau)\|_{L_{2}(\Omega)},\\
		w(\tau)&:=\frac{e^{-\mu_0 \tau}}{\tau^{\frac{n}{2}(\frac{1}{p_2'}-\frac{1}{2})}}
	\end{align*}
	for any $\tau \in [0,T]$ and extended by zero otherwise. Additionally, we require that $\frac{n}{2}(\frac{1}{p_{2}}-\frac{1}{2})<\frac{1}{h}$ and $p_1<\infty$ to ensure $\|w\|_{L_h(\mathbb{R})}<\infty$. Then we conclude that
	\begin{align*}
		&\|\delta x\|_{L_{p_1}(0,T;L_{p_2}(\Omega))}\\&\leq c\left(\|\bar{C}\delta x\|_{L_2(0,T;L_2(\Omega))} + \|B^*\delta\lambda\|_{L_2(0,T;L_2(\Omega_c))} + \|l_2\|_{L_2(0;T;L_2(\Omega))}+\|e^{-\mu t}\|_{L_{p_1}(\mathbb{R})}\|\delta x_0\|_{L_{p_2}(\Omega)}\right)\\
		&\leq c\left(\|{C}\delta x\|_{L_2(0,T;L_2(\Omega))} + \|B^*\delta\lambda\|_{L_2(0,T;L_2(\Omega_c))} + \|l_2\|_{L_2(0;T;L_2(\Omega))}+\|\delta x_0\|_{V}\right)
	\end{align*} 
	where the last estimate follows from \cref{as:nonlin:stab}, i.e., $\|\bar{C}v\|_{L_2(0,T;L_2(\Omega))}\leq \|{C}v\|_{L_2(0,T;L_2(\Omega))}$ for all $v\in L_2(0,T;L_2(\Omega))$ and by the classical Sobolev embedding theorem $H^1(\Omega)\hookrightarrow L_{p_2}(\Omega)$ for $p_2 < \frac{2n}{n-2}$, cf.\ \cite[Theorem 5.4]{Adams1975}.
\end{proof}
This stability estimate can be used to derive a $T$-uniform estimate for the solution operators norm. The latter can then be used to also bound the solution operator in exponentially scaled spaces. Both these bounds play a central role in the assumptions of the implicit function theorem \cref{thm:implfunc}. The following theorem states the main result of this section.
\begin{thrm}
	\label[thrm]{cor:nonlin:maxreg:opnorm}
	Let $z^0=(x^0,\lambda^0)$ be such that \cref{as:nonlin:Linfty} and \cref{as:nonlin:stab} hold. 
	Then, for all $2 \leq p_1,p_2$ satisfying $\frac{n}{2}(\frac{1}{2}-\frac{1}{p_2})<\frac{1}{p_1}+\frac{1}{2}$ and $p_2 < \frac{2n}{n-2}$ there is a constant $c\geq 0$ independent of $T$ such that
	\begin{align}
		\label{bound_unscaled}
		\|L_r''(z^0)^{-1}\|_{L((L_2(0,T;L_2(\Omega))\times V)^2, (L_{p_1}(0,T;L_{p_2}(\Omega))\cap W^{1,2}(0,T,D(A),L_2(\Omega))^2) } \leq c
	\end{align}
	Moreover for all $\mu > 0$ satisfying 
	\begin{align}\label{eq:muRestricton}
		\mu < \frac{1}{\|L_r''(z^0)^{-1}\|_{L((L_2(0,T;L_2(\Omega))\times V)^2,(L_{p_1}(0,T;L_{p_2}(\Omega))\cap W^{1,2}(0,T,D(A),L_2(\Omega))^2) }}
	\end{align}there is a constant $c\geq 0$ independent of $T$ such that
	\begin{align}
		\label{eq:scaledest} 
		\|L_r''(z^0)^{-1}\|_{ L((L^s_2(0,T;L_2(\Omega))\times V^{s(T)} \times L^s_2(0,T;L_2(\Omega))\times V^{s(0)}), (L^s_{p_1}(0,T;L_{p_2}(\Omega))\cap W_s^{1,2}(0,T,D(A),L_2(\Omega)))^2)} \leq c
	\end{align}
	for the scaling functions $s(t) = e^{-\mu t}$ or $s(t) = \frac{1}{e^{-\mu t}+e^{-\mu(T-t)}}$.
\end{thrm}
\begin{proof}
	Using the bound derived in \cref{thm:nonlin:almostthere}, it only remains to estimate $\|C\delta x\|^2_{L_2(0,T;L_2(\Omega))}+ \|B^*\delta\lambda\|^2_{L_2(0,T;L_2(\Omega_c))}$.
	This follows by a classical estimate in optimal control, i.e., by testing in \eqref{eq:nonlinsys} the adjoint equation with the state, the state equation with the adjoint, integrating by parts and subtracting, which yields
	\begin{align*}
		&\|C\delta x\|^2_{L_2(0,T;L_2(\Omega))}+ \|B^*\delta\lambda\|^2_{L_2(0,T;L_2(\Omega_c))} \leq \\&|\langle l_1,\delta x\rangle_{L_2(0,T;L_2(\Omega))}| + |\langle l_2,\delta \lambda \rangle_{L_2(0,T;L_2(\Omega))}| + |\langle \delta x_0,\delta\lambda(0)\rangle_{L_2(\Omega)}| + |\langle\delta \lambda_T,\delta x(T)\rangle_{L_2(\Omega)}|.
	\end{align*}
	The bound \cref{bound_unscaled} then follows. To prove the bound in the scaled spaces we proceed analogously to the proof of \cite[Theorem 3.1 and Theorem 5.2]{Gruene2018c}. Hence we define $M:=L''_r(z^0)$ and set $Z:=(L_{p_1}(0,T;L_{p_2}(\Omega))\cap W^{1,2}(0,T;D(A),L_2(\Omega)))^2$ and $E:=(L_2(0,T;L_2(\Omega))\times V)^2$. First, setting $s(t)=\frac{1}{e^{-\mu t}+e^{-\mu(T-t)}}$ a straightforward computation shows that for $\varepsilon\in E$
	\begin{align*}
		M\delta z &= \varepsilon\\
		(M-\mu P)(s\delta z)&=s\varepsilon\\
		(I-\mu M^{-1}P)(s\delta z)&=M^{-1}s\varepsilon
	\end{align*}
	where $P:=\begin{psmallmatrix}
	0&F\\
	0&0\\
	-F&0\\
	0&0
	\end{psmallmatrix}$ and $F:=\frac{(e^{-\mu(T-t)} - e^{-\mu t})}{(e^{-\mu t} + e^{-\mu(T-t)})} <1$. Thus, choosing $\mu < \frac{1}{\|M^{-1}\|_{L(E,Z)}}$ and setting $\beta = \mu \|M^{-1}\|_{L(E,Z)}<1$, a standard Neumann argument, cf.\ \cite[Theorem 2.14]{Kress1989} yields, 
	\begin{align*}
		\|s\delta z\|_Z \leq \frac{\|M^{-1}\|_{L(E,Z)}}{1-\beta}\|s\varepsilon\|_E.
	\end{align*}
	Thus, by definition of the scaled norms, the bound \cref{eq:scaledest} for $s(t)=\frac{1}{e^{-\mu t}+e^{-\mu(T-t)}}$ follows. Completely analogously we conclude \cref{eq:scaledest} for $s(t)=e^{-\mu t}$ with the same argumentation and $P:=\begin{psmallmatrix}
	0&-I\\
	0&0\\
	I&0\\
	0&0
	\end{psmallmatrix}$.
\end{proof}
We briefly comment on the estimates of \cref{thm:nonlin:almostthere} and \cref{cor:nonlin:maxreg:opnorm}.
\begin{rmrk}
	In the case of $n=2$, the restriction for $p_1,p_2$ includes all $2\leq p_1,p_2\leq \infty$ except $p_1=p_2=\infty$. If $n=3$, e.g., the choice $2\leq p_1=p_2<6$ is allowed.
	The pointwise in time estimates, i.e., choosing $p_1=\infty$ and thus requiring $p_2<\infty$ for $n=2$ and $p_2<6$ for $n=3$ are consistent with maximal parabolic regularity theory. In that case, for initial values in $H_0^1(\Omega)$ and right-hand sides in $L_2(0,T;L_2(\Omega))$ the maximal parabolic regularity theory leads to solutions continuous in time with values in $H^1_0(\Omega)$, even for non-autonomous equations, cf.\ \cite{Amann2004}. By classical embedding theorems we get $C(0,T;H^1_0(\Omega)) \hookrightarrow C(0,T;L_p(\Omega))$ with $1\leq p<\infty$ for $n=2$ and $1\leq p<6$ for $n=3$ which coincides with the choice of $p_2$ specified above.
\end{rmrk}
\begin{rmrk}
	In our setting, the differential operator $\mathcal{A}$ gives rise to an analytic semigroup on $L_l(\Omega)$ for $1\leq l< \infty$. If one only has analyticity for $1<l<\infty$, \cref{thm:nonlin:almostthere} still holds under the additional restriction that $p_2<\infty$. This restriction stems from the fact that we can not choose test functions $\psi$ that are subject to an analytic semigroup on $L_1(\Omega)$ which would be needed to obtain an $L_\infty(\Omega)$-estimate.
\end{rmrk}
\section{Exponential turnpike and sensitivity results for semilinear parabolic problems}
\label{sec:tp_sensi_semilin}
We can now combine the results of \Cref{subsec:nonlin:superpos} regarding superposition operators and the bound on the solution operator to the linearized problem of \Cref{sec:nonlin:imprreg} to apply the implicit function theorem (\Cref{thm:implfunc}) to semilinear parabolic problems. In that case, we will choose $p_1=p_2$ as the image space of the superposition operators is $L_2(0,T;L_2(\Omega))$, i.e., spatial and temporal integrability coincide. In that case, the assumptions of \cref{cor:nonlin:maxreg:opnorm} on $p_1=p_2=p$ simplify to $p<\frac{2n}{n-2}$. This choice of $p$ represents the highest exponent of all $L_p(\Omega)$-spaces in which $H^1(\Omega)$ is embedded into.
\begin{thrm}
	\label[thrm]{prop:nonlin:maxreg:continuous}
	Let $f(x)$ and $J_x(x)$ induce twice continuously Fr\'echet differentiable superposition operators from $L_p(0,T;L_p(\Omega))$ to $L_2(0;T;L_2(\Omega))$ with $T$-uniformly continuous derivatives. Suppose the assumptions of \cref{cor:nonlin:maxreg:opnorm} hold with $p_1=p_2=p$, i.e., $2 \leq p<\frac{2n}{n-2}$. Then $G(z,\varepsilon):=L_r'(z)-\varepsilon$ with $L_r'$ given in \eqref{eq:nonlin:exactextremal_red} satisfies the assumptions of \cref{thm:implfunc} uniformly in $T$ for any $\mu > 0$ satisfying \eqref{eq:muRestricton} and setting either
	\begin{enumerate}[i)]
		\item $z^0=(\bar{x},\bar{\lambda})$ solving the steady-state problem \eqref{eq:nonlin:staticextremal} and $\varepsilon^0=(0,\bar{\lambda},0,\bar{x}-x_0)$ and the scaling $s(t)=\frac{1}{e^{-\mu t}+e^{-\mu(T-t)}}$, or
		\item $z^0=({x},{\lambda})$ solving the dynamic problem \eqref{eq:nonlin:exactextremal_red} and $\varepsilon^0=0$ and the scaling $s(t)=e^{-\mu t}$
	\end{enumerate} in the spaces $Z=\left(L_{p}(0,T;L_{p}(\Omega))\cap W^{1,2}(0,T,D(A),L_2(\Omega))\right)^2$ and $E=\left(L_2(0,T;L_2(\Omega))\times V\right)^2$ endowed with the scaled norms $\|\cdot\|_{Z_s}=\|\cdot\|_{\left(L^s_{p}(0,T;L_{p}(\Omega))\cap W_s^{1,2}(0,T,D(A),L_2(\Omega))\right)^2}$ and $\|\cdot\|_{E_s}=\|\cdot\|_{L^s_2(0,T;L_2(\Omega))\times V^{s(T)}\times L^s_2(0,T;L_2(\Omega))\times V^{s(0)}}$.
\end{thrm}
\begin{proof}
	Assumption ii) and iii) of \cref{thm:implfunc} follow by $T$-uniform continuity of the derivative of the superposition operator via \cref{lem:nonlin:scaleddiff}. $T$-uniform continuous invertibility of $L_{r}''(\bar{x},\bar{\lambda})^{-1}$ resp.\ $L_{r}''(\bar{x},\bar{\lambda})^{-1}$ in scaled and unscaled spaces follows from \cref{cor:nonlin:maxreg:opnorm}.
\end{proof}
This result can now be used to deduce a local turnpike result, stating that solutions of the dynamic problem \cref{eq:nonlin:exactextremal} are close to solutions of the static problem \cref{eq:nonlin:staticextremal} for the majority of the time under the assumption that initial resp.\ terminal values are close enough at the turnpike.
\begin{crllr}
	\label[crllr]{thm:nonlin:maxreg:turnpike}
	Let the assumptions of \cref{prop:nonlin:maxreg:continuous} hold with $2\leq p< \frac{2n}{n-2}$. Consider $(x,u,\lambda)$ solving the nonlinear dynamic problem \eqref{def:OCP} and $(\bar{x},\bar{u},\bar{\lambda})$ solving the nonlinear static problem \eqref{def:SOCP}. 
	Define $(\delta x,\delta u,\delta \lambda) := (x-\bar{x},u-\bar{u},\lambda-\bar{\lambda})$. Then there are $r_E>0$, $\mu >0$ (satisfying  \cref{eq:muRestricton}) and $c\ge 0$ independent of $T$ such that if 
	\begin{align*}
		\|x_0-\bar{x}\|_{V}+\|\bar{\lambda}\|_{V}\leq r_E
	\end{align*}
	it holds that
	\begin{align*}
		\norm{\frac{1}{e^{-\mu t}+e^{-\mu(T-t)}}\delta x(t)}_{L_{p}(0,T;L_{p}(\Omega))\cap W^{1,2}(0,T,D(A),L_2(\Omega))} + &\norm{\frac{1}{e^{-\mu t}+e^{-\mu(T-t)}}\delta u(t)}_{L_\infty(0,T;L_2(\Omega_c))} \\
		+\norm{\frac{1}{e^{-\mu t}+e^{-\mu(T-t)}}\delta \lambda(t)}_{L_{p}(0,T;L_{p}(\Omega))\cap W^{1,2}(0,T,D(A),L_2(\Omega))} &\leq cr_E.
	\end{align*}
\end{crllr}
\begin{proof}
	The proof follows by  \cref{prop:nonlin:maxreg:continuous} with the choice i) for the scaling function and applying \cref{thm:implfunc}.
\end{proof}
Second, we can conclude a sensitivity result, which states that perturbations of the extremal equations' dynamics that are small at an initial part lead to disturbances in the variables that are small at an initial part. More specifically we obtain that solutions to the perturbed dynamic problem \cref{eq:nonlin:perturbedextremal} are close to the solutions of the unperturbed dynamic problem \cref{eq:nonlin:exactextremal} on an initial part, even if the perturbations increase exponentially. In that context, we have to assume that the perturbations in unscaled norms are sufficiently small.
\begin{crllr}
	\label[crllr]{thm:nonlin:maxreg:sensitivity}
	Let the assumptions of \cref{prop:nonlin:maxreg:continuous} hold with $2\leq p<\frac{2n}{n-2}$. Let $(x,\lambda)$ solve the nonlinear extremal equations \eqref{eq:nonlin:exactextremal_red} and $(\tilde{x},\tilde{\lambda})$ the perturbed extremal equations \eqref{eq:nonlin:perturbedextremal}.
	Define $(\delta x,\delta \lambda) := (\tilde{x}-x,\tilde{\lambda}-\lambda)$ and $\delta u = Q^{-1}B^*\delta \lambda$. Then there are $r_E>0$, $\mu > 0$ satisfying \eqref{eq:muRestricton} and $c\geq 0$ independent of $T$ such that if 
	\begin{align*}
		\|\varepsilon_1\|_{L_p(0,T;L_p(\Omega))}+\|\varepsilon_T\|_{V}+\|\varepsilon_2\|_{L_2(0,T;L_2(\Omega))}+\|\varepsilon_0\|_{V} \leq r_E
	\end{align*}
	and setting
	\begin{align*}
		\rho:=\|e^{-\mu t}\varepsilon_1\|_{L_2(0,T;L_2(\Omega))}+\|e^{-\mu T}\varepsilon_T\|_{V}+\|e^{-\mu t}\varepsilon_2\|_{L_2(0,T;L_2(\Omega))}+\|\varepsilon_0\|_{V}
	\end{align*}
	it holds that
	\begin{align*}
		\norm{e^{-\mu t}\delta x(t)}_{L_{p}(0,T;L_{p}(\Omega))\cap W^{1,2}(0,T,D(A),L_2(\Omega))} &+ \norm{e^{-\mu  t}\delta u(t)}_{L_\infty(0,T;L_2(\Omega_c))}
		\\ + \norm{e^{-\mu  t}\delta \lambda(t)}&_{L_{p}(0,T;L_{p}(\Omega))\cap W^{1,2}(0,T,D(A),L_2(\Omega))} \leq c\rho.
	\end{align*}
\end{crllr}
\begin{proof}
	The proof follows by  \cref{prop:nonlin:maxreg:continuous} with the choice ii) for the scaling function and applying \cref{thm:implfunc}.
\end{proof}
\begin{rmrk}
	\label[rmrk]{rem:bound}
	We assumed in this part that the control operator is bounded as linear operator to $L_2(\Omega)$, ruling out the case of  boundary control. The case of boundary control could be included if one can ensure that the closed-loop semigroup is analytic and satisfies the stability estimate \cref{prop:extensionestimate}. Perturbations of analytic semigroups can be analyzed with the notion of $A$-boundedness or $A$-compactness, cf. \cite[Chapter III]{Engel2000}.
\end{rmrk}
\begin{xmpl}
	\label[xmpl]{ex:nonlin}
	We present an example with distributed control of a heat equation with Dirichlet boundary conditions. To this end we set $V=H^1_0$, $\bar{B}=\chi_{\Omega_c}$, where $\Omega_c\subset \Omega$ non-empty, $\mathcal{A}=\Delta$ and a static reference $x_\text{d}=\bar{x}_d\in L_2(\Omega)$.
	Let the cost functional be given by $J(x,u)=\frac{1}{2}\int_0^T\|x-x_\text{d}\|_{L_2(\Omega)}^2 + \|u\|^2_{L_2(\Omega_c)}$, $d=0$ and consider the nonlinearity $f(x)=x^3-c_0x$ for $c_0\in \mathbb{R}$. If $c_0$ is larger than the smallest eigenvalue of $-\Delta$ in $H^1_0(\Omega)$, the uncontrolled PDE is unstable.
	Moreover, using maximal elliptic regularity, cf.\ \cite{Casas1993} we get for the solution of the static system
	\begin{align*}
		\|\bar{x}\|_{L_\infty(\Omega)}+\|\bar{\lambda}\|_{L_\infty(\Omega)} \leq c(\Omega)\|x_\text{d}\|_{L_2(\Omega)},
	\end{align*} i.e., $(\bar{x},\bar{\lambda})$ satisfy \cref{as:nonlin:Linfty}. 
	Thus, choosing $x_\text{d}$ sufficiently small such that $\|\bar{\lambda} \bar{x}\|_{L_\infty(0,T;L_\infty(\Omega))}=\|\bar{\lambda} \bar{x}\|_{L_\infty(\Omega)}=\underline{m}<1$, the operator 
	\begin{align*}
		(L_r)_{xx}(\bar{x},\bar{\lambda}) = I - \bar{\lambda} \bar{x}.
	\end{align*}
	is nonnegative and satisfies the assumptions of \cref{ass:ssc}.
	The square root $\bar{C}$ of this operator in the sense of \eqref{eq:CstarC} can thus be defined pointwise for $v\colon \Omega \to \mathbb{R}$ and a.e.\ $\omega \in \Omega$ by \begin{align*}
		({C}v)(\omega) = \sqrt{(1-\bar\lambda(\omega)\bar x(\omega))}v(\omega).
	\end{align*}
	As the linearization point is the turnpike, i.e., a steady-state, ${C}$ as defined in \cref{eq:CstarC} is time-independent and we can set $\bar{C}={C}$ in \cref{as:nonlin:stab}. The remaining parts of \cref{as:nonlin:stab} can be verified with the generalized Poincar\'e inequality, cf.\ \cite[Lemma 2.5]{Troeltzsch2010}, analogously to \cite[Example 3.7]{Gruene2018c}. Hence, \cref{as:nonlin:stab} is satisfied and we can apply the turnpike result of \cref{thm:nonlin:maxreg:turnpike}. 
	
	In order to apply the sensitivity result of \cref{thm:nonlin:maxreg:sensitivity} we need to analyze
	\begin{align*}
		(L_r)_{xx}(x,\lambda) = I-x\lambda
	\end{align*}
	where $(x,\lambda)$ solves \eqref{eq:nonlin:exactextremal_red}. In \cite{Raymond1999}, the authors deduce a $T$-dependent bound 
	\begin{align*}
		\|x\|_{L_{\infty}(0,T;L_{\infty}(\Omega))} + &\|\lambda\|_{L_{\infty}(0,T;L_{\infty}(\Omega))} \\\leq &c(T)\left(\|x_\text{d}\|_{L_{2+\delta}(0,T;L_{2+\delta}(\Omega))}+\|u_\text{d}\|_{L_{2+\delta}(0,T;L_{2+\delta}(\Omega_c))}+\|x_0\|_{L_\infty(\Omega)}\right).
	\end{align*}
	for any $\delta > 0$. If the nonlinear and linearized uncontrolled equation is stable, e.g., if $c_0\geq 0$, it is possible to show the above bound for the state independently of $T$, cf.\ \cite[Lemma 1.1]{Pighin2020} where such an estimate was shown under the assumption that $x_d\in L_\infty(0,T;L_\infty(\Omega))$. Having bounded the state the corresponding bound on the adjoint can be obtained, cf.\ \cite[Lemma A.1]{Pighin2020} by parabolic regularity.
	Hence, choosing the data $x_d$, $u_d$ and $x_0$ small enough, similar to the elliptic case, we have $\|\lambda x\|_{L_\infty(0,T;L_\infty(\Omega))}=\underline{m}<1$ and thus the operator $(L_r)_{xx}(x,\lambda)$ satisfies the assumptions of \cref{ass:ssc} and and we can define the square root $C$ via \eqref{eq:CstarC} with $\|C\|_{L(L_p(0,T;L_q(\Omega))} = \|\sqrt{(1-\lambda x)}\|_{L_\infty(0,T;L_\infty(\Omega))}$ for all $1\leq p,q\leq \infty$.
	Choosing $\bar{C}=\underline{c}I$ in \cref{as:nonlin:stab}, where $\underline{c}\leq \sqrt{1-\underline{m}^2}$ yields for $v\colon (0,T)\times \Omega \to \mathbb{R}$ and a.e.\ $(t,\omega) \in (0,T)\times \Omega$ the estimate \begin{align*}
		\underline{c}^2|v(t,\omega)|\leq|1-\underline{m}||v(t,\omega)|= \left(1-\|\lambda x\|_{L_\infty(0,T;L_\infty(\Omega))}\right)|v(t,\omega)|\leq\|1-\lambda x\|_{L_\infty(0,T;L_\infty(\Omega))}|v(t,\omega)|.
	\end{align*} 
	Taking the square root yields $\|\bar{C}v\|_{L_2(0,T;L_2(\Omega))}<\|Cv\|_{L_2(0,T;L_2(\Omega))}$ for all $v\in L_2(\Omega)$. Thus, \cref{as:nonlin:stab} is satisfied and we can apply the sensitivity result of \cref{thm:nonlin:maxreg:sensitivity}.
	
\end{xmpl}
We will briefly discuss the smallness assumptions made in \cref{ex:nonlin} for the application of the turnpike theorem \cref{thm:nonlin:maxreg:turnpike}. First, in order to obtain a square root of $(L_r)_{xx}(\bar{x},\bar{\lambda})$, we hinge on smallness of $\bar{x}$ and $\bar{\lambda}$ in $L_\infty(\Omega)$. Second, for the application of \cref{thm:nonlin:maxreg:turnpike}, we have to assume smallness of $\|x_0-\bar{x}\|_{V}$ and $\|\bar{\lambda}\|_V$ where $V$ is not necessarily embedded in $L_\infty(\Omega)$. The combination of smallness in these norms allows to conclude a turnpike result. 

We briefly discuss the relation of the established turnpike result and \cref{ex:nonlin} to the turnpike results of \cite{Porretta2016} and \cite{Trelat2016}. Under smallness assumptions on $\|y_0-\bar{y}\|_{L_\infty(\Omega)}$ and $\|\bar{\lambda}\|_{L_\infty(\Omega)}$, a turnpike result in $C(0,T;L_\infty(\Omega))$ for semilinear heat equations was given in \cite[Theorem 1]{Porretta2016} for space dimension $n\leq 3$. 
Further in \cite[Theorem 1]{Trelat2016}, for a Hilbert space $V$, $C(0,T;V)$-estimates are concluded via smallness assumptions of the initial and terminal distance to the turnpike in $V$. In this case, positive semi-definiteness of $(L_r)_{xx}$ is assumed in order to obtain a square root. This was then verified a posteriori for an example by assuming smallness of the optimal steady state, similar to our approach in \cref{ex:nonlin}.
Here, we assume in \cref{thm:nonlin:maxreg:turnpike} smallness in $H^1(\Omega)$ (which does not necessarily imply smallness in $L_\infty(\Omega)$ if $n\geq 2$). However, as seen in the previous example, we additionally need smallness in $L_\infty(\Omega)$ in order to obtain a square root of $(L_r)_{xx}$. Under these assumptions, by invoking  \cref{thm:nonlin:maxreg:turnpike} we obtain $L_p(0,T;L_p(\Omega))$-estimates for $2\leq p <\frac{2n}{n-2}$ which are weaker than the $L_\infty(\Omega)$ estimates of \cite{Porretta2016}. However, we also obtain estimates in the Sobolev norm $W^{1,2}(0,T;D(A),L_2(\Omega))$ which yields turnpike of time and space derivatives. This is a novelty in this work regarding turnpike properties. Compared to \cite{Trelat2016}, we obtain the same estimates, i.e., smallness in $V$ implies turnpike in $W^{1,2}(0,T;D(A),L_2(\Omega))\hookrightarrow C(0,T;V)$.

\begin{rmrk}
	\label[rmrk]{rem:nonlin:nonlincost}
	We briefly discuss the case of nonlinearities that are sums of monotone polynomials, e.g., $f(x)=x^3+x^5$. In standard applications of superposition operators where the estimates do not need to be uniform in the size of the domain, only the behavior of the nonlinearity towards infinity is important. Thus, in case of $f(x)=x^3+x^5$ one would estimate the cubic term on the set where $x>1$ by the higher order term $x^5$ and bound the remainder by the measure of the domain, as there $x\leq 1$. This is not possible if one is particularly interested in estimates independent of the size of the domain, i.e., in our case, independent of $T$, cf.\ also \Cref{ex:revisited}. As a remedy, one has to invoke \cref{thm:implfunc} with the space $Z=\left(L_{10}(0,T;L_{10}(\Omega))\cap L_{{6}}(0,T;L_{{6}}(\Omega))\cap W^{1,2}(0,T;D(A),L_2(\Omega))\right)^2$, where the bound on the solution operator follows by \cref{cor:nonlin:maxreg:opnorm} if $n=2$. This means that the turnpike estimate in the Sobolev spaces remains unchanged. However, one obtains additionally estimates in other $L_p(0,T;L_p(\Omega))$-spaces. Note that these estimates are not equivalent independently of $T$, i.e., we can not conclude a turnpike in $L_6(0,T;L_6(\Omega))$ from a turnpike in $L_{10}(0,T;L_{10}(\Omega))$ as the constant would depend on $T$.
\end{rmrk}
\section{Numerical examples}
\label{sec:nonlin:numerics}
In this part, we will showcase the theoretical results from \Cref{sec:nonlin:imprreg} by means of an example of a semilinear heat equation. We further present numerical results for the boundary control of a quasilinear heat equation. All numerical examples are performed with the C++-library for vector space algorithms \textit{Spacy}\footnote{https://spacy-dev.github.io/Spacy/}  using the finite element library \textit{Kaskade7} \cite{Goetschel2020}.

We consider $T=10$, $\Omega=[0,3]\times[0,1]$, and the cost functional
\begin{align*}
	J(x,u):= \frac{1}{2}\|(x-x_\text{d})\|^2_{L_2(0,T;L_2(\Omega))} + \frac{\alpha}{2} \|u\|^2_{L_2(0,T;L_2(\Omega))},
\end{align*}
where $x_\text{d}$ is either of the following: To illustrate the turnpike property, we will use the static reference defined by
\begin{align}
	\label{def:parabolic:reference_auto}
	x_\text{d}^\text{stat}(\omega) &:= g\left(\frac{10}{3}\norm{\omega - \begin{pmatrix}
			1.5\\0.5\end{pmatrix}}\right),\\
	\text{where}\quad\qquad g(s) &:= \begin{cases}
		10e^{1-\frac{1}{1-s^2}} \qquad &s< 1\\
		0   &\text{else}.
	\end{cases}
\end{align}
We proved in \cref{thm:nonlin:maxreg:sensitivity} that perturbations that occur far in the future have a negligible effect on the initial part of the optimal triple. Motivated by this, we will suggest two a priori discretization schemes tailored to an MPC context. To evaluate the performance of these schemes, we will also consider non-autonomous problems. 
To this end, we define, with $g(s)$ as above,
\begin{align}
	\label{def:parabolic:reference_nonauto}
	x_\text{d}^\text{dyn}(t,\omega) &:= g\left(\frac{10}{3}\norm{\omega - \begin{pmatrix}
			\omega_{1,\text{peak}}(t)\\\omega_{2,\text{peak}}(t)\end{pmatrix}}\right)
\end{align}
and 
\begin{align*}
	\omega_{1,\text{peak}}(t):= 1.5 - \cos\left( \pi\left( \frac{t}{10} \right)\right),\qquad 
	w_{2,\text{peak}}(t):= \left\vert \cos\left(\pi\left(\frac{t}{10}\right)\right)\right\vert,
\end{align*}

We will consider the MPC \cref{alg::mpcabstract} and evaluate the performance of different a priori time grids used in every solution of the OCP.
\begin{lgrthm}[Standard MPC Algorithm]\hfill
	\label[lgrthm]{alg::mpcabstract}
	\begin{algorithmic}[1]
		\STATE{Given: Prediction horizon $0<T$, implementation horizon $0<\tau\leq T$, initial state $x_0$}
		\STATE{$k=0$}
		\WHILE{controller active}
		\STATE{Solve OCP on $[k\tau,T+k\tau]$ with initial datum $x_k$, save optimal control in $u$}
		\STATE{Implement $u_{\big|[k\tau,(k+1)\tau]}$ as feedback, measure/estimate resulting state and save in $x_{k+1}$}
		\STATE{$k = k+1$}
		\ENDWHILE
	\end{algorithmic}
\end{lgrthm}
We now present two specialized a priori discretization techniques motivated by the exponential decay of perturbations deduced in \cref{thm:nonlin:maxreg:sensitivity}.
First, we will evaluate an exponential distribution of grid points. To this end, we compute vertices $t_i,\, i\in \{0,\ldots,N-1\}$ such that
\begin{align}
	\label{gridscaling}
	\int\limits_{t_i}^{t_{i+1}}e^{-ct}\,dt  = \frac{I}{N-1} \qquad \forall i\in \{0,\ldots,N-2\},
\end{align}
where $I:=\int_0^Te^{-ct}\,dt=\frac{1}{c}\left(1-e^{-cT}\right)$. For all computations considered in this work, we chose $c=1$.

As a second specialized refinement procedure we propose to use the same number of grid points on $[0,\tau]$ as on $[\tau,T]$. If $\tau \ll T$, this naturally leads to a finer mesh on the initial part. As a reference, we use a standard uniform grid. All three gridding schemes are depicted in the following for eleven grid points.
\begin{enumerate}[i)]
	\item Uniform:
	
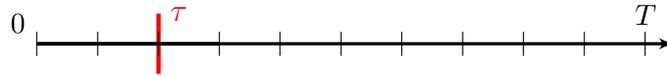
\begin{center}
		\begin{tikzpicture}[scale = 0.8]
		\draw [very thick](3,0) -> (0,0) node [above left] {0};
		\draw [red,ultra thick](2,-0.5)--(2,0.5) node [right] {\large$\tau$};
		\draw [arrow,very thick](0,0) -> (10.5,0) node [above left=0.1cm] {$T$};
		\draw [](0,0.2) -- (0,-0.2) node [below right] {};
		\draw [](1,0.2) -- (1,-0.2) node [below right] {};
		\draw [](2,0.2) -- (2,-0.2) node [below right] {};
		\draw [](3,0.2) -- (3,-0.2) node [below right] {};
		\draw [](4,0.2) -- (4,-0.2) node [below right] {};
		\draw [](5,0.2) -- (5,-0.2) node [below right] {};
		\draw [](6,0.2) -- (6,-0.2) node [below right] {};
		\draw [](7,0.2) -- (7,-0.2) node [below right] {};
		\draw [](8,0.2) -- (8,-0.2) node [below right] {};
		\draw [](9,0.2) -- (9,-0.2) node [below right] {};
		\draw [](10,0.2) -- (10,-0.2) node [below right] {};
		\end{tikzpicture}
	\end{center}
	\item Exponential:
	
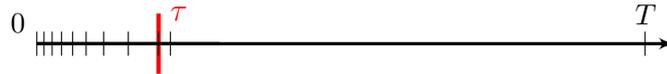
\begin{center}
		\begin{tikzpicture}[scale = 0.8]
		\draw [very thick](3,0) -> (0,0) node [above left] {0};
		\draw [red,ultra thick](2,-0.5)--(2,0.5) node [right] {\large$\tau$};
		\draw [arrow,very thick](0,0) -> (10.5,0) node [above left=0.1cm] {$T$};
		\draw [](0,0.2) -- (0,-0.2) node [below right] {};
		\draw [](0.117777,0.2) -- (0.117777,-0.2) node [below right] {};
		\draw [](0.251301,0.2) -- (0.251301,-0.2) node [below right] {};
		\draw [](0.405442,0.2) -- (0.405442,-0.2) node [below right] {};
		\draw [](0.58775,0.2) -- (0.58775,-0.2) node [below right] {};
		\draw [](0.810873,0.2) -- (0.810873,-0.2) node [below right] {};
		\draw [](1.09852,0.2) -- (1.09852,-0.2) node [below right] {};
		\draw [](1.50392,0.2) -- (1.50392,-0.2) node [below right] {};
		\draw [](2,0.2) -- (2,-0.2) node [below right] {};
		\draw [](2.19686,0.2) -- (2.19686,-0.2) node [below right] {};
		\draw [](10,0.2) -- (10,-0.2) node [below right] {};
		\end{tikzpicture}
	\end{center}
	\item Piecewise uniform:
	
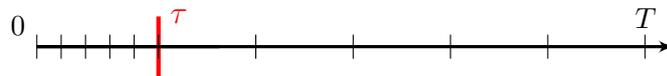
\begin{center}
		\begin{tikzpicture}[scale = 0.8]
		\draw [very thick](3,0) -> (0,0) node [above left] {0};
		\draw [red,ultra thick](2,-0.5)--(2,0.5) node [right] {\large$\tau$};
		\draw [arrow,very thick](0,0) -> (10.5,0) node [above left=0.1cm] {$T$};
		\draw [](0,0.2) -- (0,-0.2) node [below right] {};
		\draw [](0.4,0.2) -- (0.4,-0.2) node [below right] {};
		\draw [](0.8,0.2) -- (0.8,-0.2) node [below right] {};
		\draw [](1.2,0.2) -- (1.2,-0.2) node [below right] {};
		\draw [](1.6,0.2) -- (1.6,-0.2) node [below right] {};
		\draw [](2,0.2) -- (2,-0.2) node [below right] {};
		\draw [](3.6,0.2) -- (3.6,-0.2) node [below right] {};
		\draw [](5.2,0.2) -- (5.2,-0.2) node [below right] {};
		\draw [](6.8,0.2) -- (6.8,-0.2) node [below right] {};
		\draw [](8.4,0.2) -- (8.4,-0.2) node [below right] {};
		\draw [](10,0.2) -- (10,-0.2) node [below right] {};
		\end{tikzpicture}
	\end{center}
\end{enumerate}

\subsection{Distributed control of a semilinear heat equation}
We consider dynamics governed by the semilinear heat equation
\begin{align*}
	&&x' - 0.1\Delta x + ex^3 &= u  &&\text{in }\Omega \times (0,T),\\
	&&x &= 0 &&\text{in }\partial\Omega \times (0,T),\\
	&&x(0) &= 0 &&\text{in } \Omega,
\end{align*}
where $e\geq 0$ is a nonlinearity parameter.
We first illustrate the turnpike property and hence choose a static reference trajectory, i.e., $x_\text{d}^\text{stat}$. In \cref{fig:parabolic:semilin:openloopnonautonomous} the norm of the optimal state and control for different nonlinearity parameters $e$ are depicted. The turnpike property emerges in all four cases, even for very high choices of the nonlinearity parameter. Additionally, we observe that the norm of the turnpike decreases for increasing nonlinearity. This is due to the fact that the nonlinearity forces the solution of the state equation to zero, which can be seen by testing the state equation with the state, integrating by parts in time and space and using the Poincar\'e inequality which leads to
\begin{align*}
	\|x(t)\|^2 \leq  -c(\Omega)\int_0^t\|x(s)\|^2 +  e\|x(s)\|^4\,ds + \int_0^t u(s)x(s)\,ds.
\end{align*}
We note that the depicted plots only show the turnpike property for a space-time discretized problem. For further numerical experiments that indicate that the turnpike also holds in function space, the interested reader is referred to \cite{Gruene2020}.
\begin{figure}[H]
	\centering
	{\input{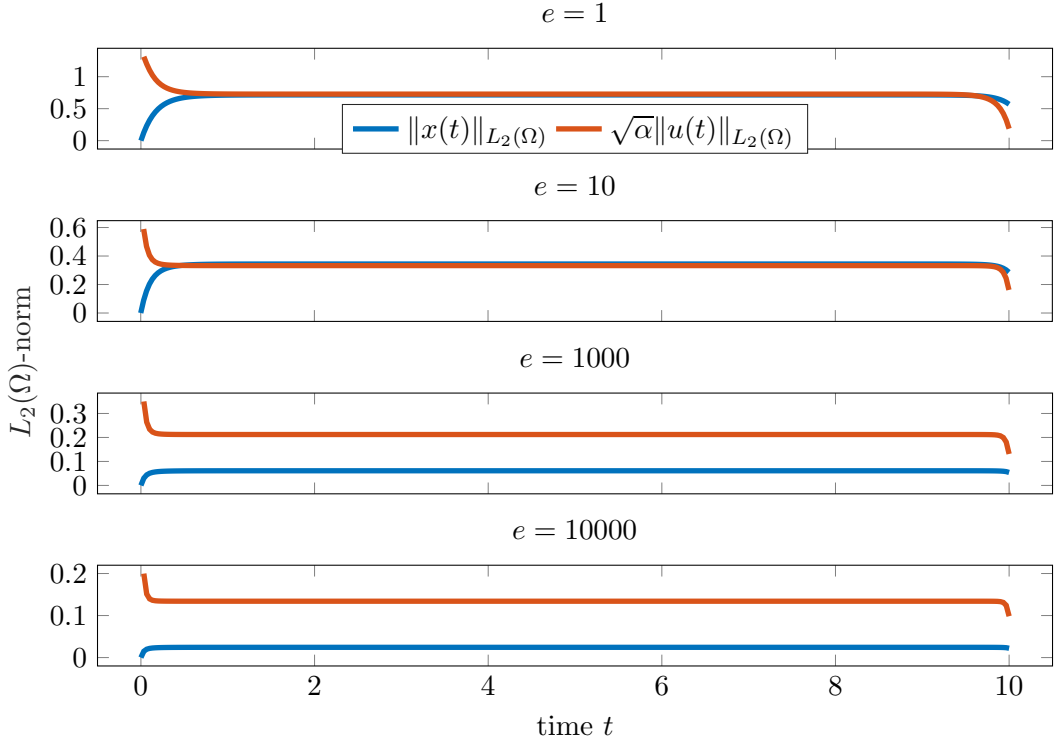}}
	\caption[Turnpike property in case of a semilinear equation]{Spatial norm of open loop state and control over time with the static reference $x_\text{d}^\text{stat}$ and $\alpha = 10^{-1}$ for different nonlinearity parameters.}
	\label{fig:parabolic:semilin:openloopnonautonomous}
\end{figure}
Second, we apply 4 steps of the MPC algorithm \cref{alg::mpcabstract} to the optimal control problem above. We set the implementation horizon $\tau = 1$ and choose the dynamic reference $x_\text{d}^\text{dyn}$ defined in \eqref{def:parabolic:reference_nonauto}. The simulation of the closed-loop trajectory emerging from the MPC feedback is again computed on three uniform refinements of the initial grid. \cref{fig:parabolic:semilin:funcvals} shows the closed-loop cost for different a priori time discretization regimes. It can be seen that the exponential and piecewise uniform time grids achieve lower closed-loop cost than a conventional uniform grid. As all grids are constructed a priori, we note again that the numerical effort is the same for all three techniques, as the grids are generated a priori.

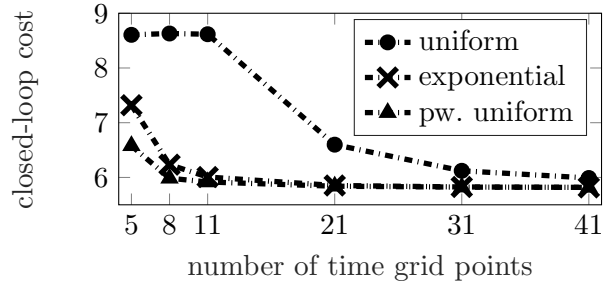
\begin{figure}[H]
	\centering
	{\input{figs/semilin_fvals.tex}}
	\caption[Comparison of MPC closed-loop cost for a priori time discretization of an non-autonomous semilinear problem]{Comparison of MPC closed-loop cost for a different priori time discretization regimes with dynamic reference $x_\text{d}^\text{dyn}$ and parameters $e=1$ and $\alpha =10^{-2}$.}
	\label{fig:parabolic:semilin:funcvals}
\end{figure}
\subsection{Boundary control of a quasilinear equation}
\label{subsec:nonlin:numerics:quasilin}
As a second numerical example, we consider a heat equation with heat conductivity depending on the temperature. Again, we refer the reader to \cite{Gruene2020} for various numerical examples with space-time grid adaptivity that would go beyond the scope of this work. To this end, we introduce the heat conduction tensor
\begin{align*}
	\kappa(x)(t,\omega):=\left(c|x(t,\omega)|^2+0.1\right),
\end{align*}
where $c \geq 0$ is a nonlinearity parameter and consider the quasilinear dynamics
\begin{align*}
	&&x' - \nabla \cdot(\kappa(x)\nabla x) &= 0  &&\text{in }\Omega \times (0,T),\\
	&&\kappa(x)\frac{\partial x}{\partial\nu} &= u &&\text{in }\partial\Omega \times (0,T),\\
	&&x(0) &= 0 &&\text{in } \Omega.
\end{align*}
Our theoretical results of \Cref{sec:nonlin:imprreg} do not cover the case of a quasilinear equation. However, the turnpike property can be observed in \cref{fig:parabolic:quasilin:openloopnonautonomous} even for very large choices of the nonlinearity parameter $c$. Moreover, we observe the same behavior of the norm of the turnpike as in the semilinear example: for increased nonlinearity, the norm of the turnpike decreases. This again reflects the effect of the nonlinearity forcing the state to zero. We depict the turnpike property in a norm that is motivated by the second derivative of the Lagrangian, i.e., a scaled $H^1(\Omega)$-norm $\|v\|_{\alpha d,H^1(\Omega)}:=\|v\|_{L_2(\Omega)}+\sqrt{d\alpha}\|\nabla v\|_{L_2(\Omega)}$ with $d=0.1$. 
\begin{figure}[H]
	\centering
	{\input{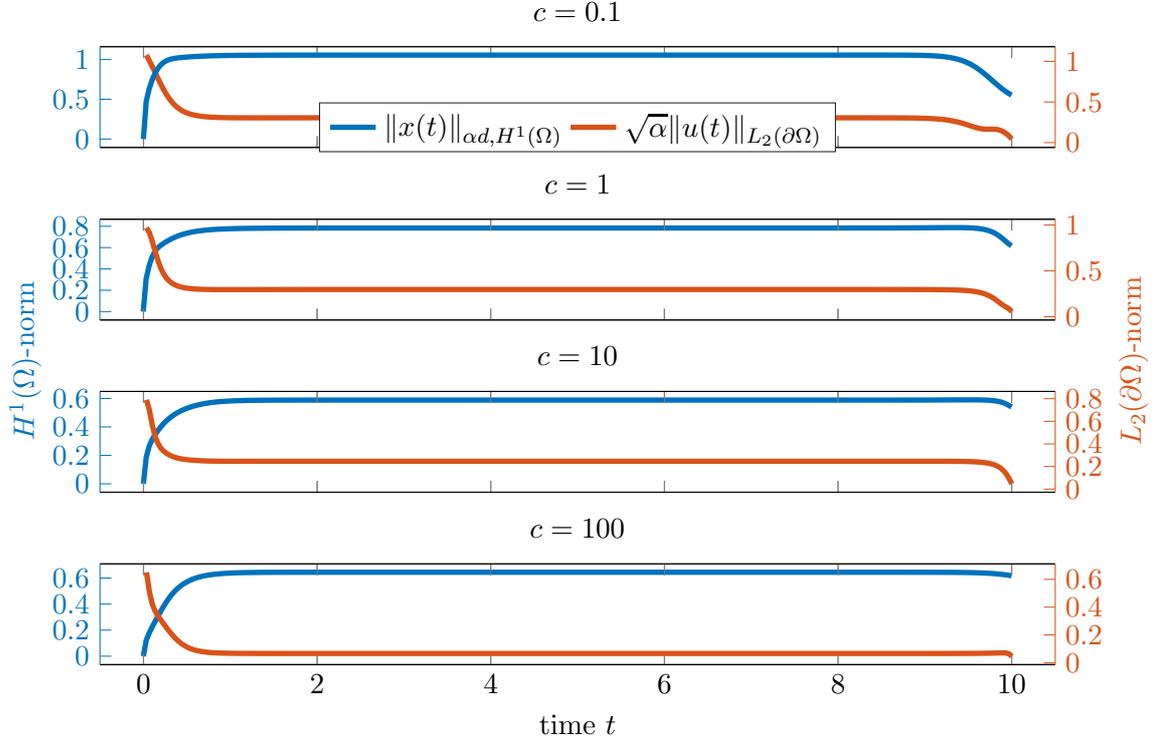}}
	\caption[Turnpike property in case of a boundary controlled quasilinear equation]{Spatial norm of open loop state and control over time with the static reference $x_\text{d}^\text{stat}$ for different nonlinearity parameters and $\alpha = 10^{-1}$.}
	\label{fig:parabolic:quasilin:openloopnonautonomous}
\end{figure}
In \cref{fig:parabolic:quasilinlin:funcvals} we compare the closed-loop cost of different a priori time discretization regimes. Similar to the semilinear example investigated before, we observe that exponential and piecewise uniform a priori time grids outperform the conventional uniform grid.
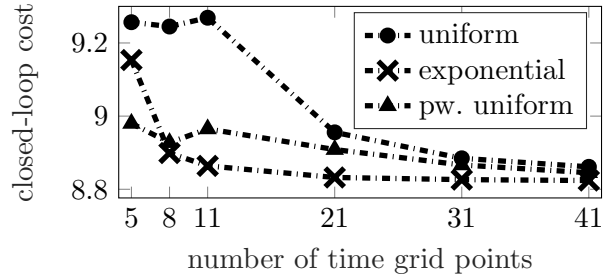
\begin{figure}[H]
	\centering
	{\input{figs/quasilin_fvals.tex}}
	\caption[Comparison of MPC closed-loop cost for a priori time discretization of an non-autonomous quasilinear problem]{Comparison of MPC closed-loop cost for a different priori time discretization regimes with dynamic reference $x_\text{d}^\text{dyn}$ for different priori time discretization regimes with parameters $c=0.1$ and $\alpha =10^{-2}$.}
	\label{fig:parabolic:quasilinlin:funcvals}
\end{figure}
\section{Conclusion and outlook}
We have proposed an abstract approach to analyze the sensitivity of necessary optimality conditions arising in optimal control. To this end we presented an implicit function theorem that, under $T$-independent continuity and invertibility assumptions on the linearization, allows to deduce estimates in scaled norms. We further applied the abstract theory to the case of optimal control of a class of semilinear heat equations to derive a turnpike property and a sensitivity result stating that perturbations of the dynamics decay exponentially in time. Finally we presented numerical examples that illustrate these theoretical results. 
We conclude this paper by presenting some possible directions of further research. 

First, a nonlinear dependence on the control could be introduced. In our case, the quadratic dependence of the cost functional on the control allowed for a direct elimination of the control. Additionally, we did not have deal with superposition operators for the control, where improved regularity of the optimal control might be needed. In particular cases this improved regularity can be established by classical bootstrapping. If the system depends nonlinearly on the control, a standard assumption is the existence of $\alpha > 0$ such that $J_{uu}(x,u)(\delta u,\delta u)\geq \alpha \|\delta u\|^2_{L_2(0,T;U)}$ for all $\delta u\in L_2(0,T;U)$ in order to represent the optimal control by the adjoint state arising in the first order necessary conditions. This property is sometimes referred to as the strengthened Legendre-Clebsch condition, cf.\ \cite[Chapter 6]{Bryson1975}. 

Second, as discussed in \cref{rem:bound}, another natural extension would be the case of boundary control.

Finally, the abstract approach of \Cref{subsec:nonlin:implfunc} could be used to derive turnpike or sensitivity estimates for nonlinear hyperbolic or elliptic systems.

\bibliographystyle{abbrv}
\bibliography{references.bib}
\end{document}

%% file: figs/semilin_fvals.tex
%
%
\definecolor{mycolor1}{rgb}{0.00000,0.44700,0.74100}%
\begin{tikzpicture}

\begin{axis}[%
width=2.5in,
height=1in,
at={(0in,0in)},
scale only axis,
xmin=4,
xmax=42,
xtick={ 5,  8, 11, 21, 31,41},
xlabel style={font=\color{white!15!black}},
xlabel={number of time grid points},
ymax=9,
ylabel style={font=\color{white!15!black}},
ylabel={closed-loop cost},
axis background/.style={fill=white},
title style={align=center,at={(2.6in,1)}},
legend style={legend pos = north east,legend cell align=left, align=left, draw=white!15!black}
]
\addplot [color=black, dashdotted, line width=2.0pt, mark=*, mark options={solid, black}]
  table[row sep=crcr]{%
5 8.60578546137194\\
8 8.62995872254623\\
11 8.61692437263338\\
21 6.59876982067082\\
31 6.12149267683271\\
41 5.98997255675213\\
};

\addlegendentry{uniform}

\addplot [color=black, dashdotted, line width=2.0pt, mark size=5.0pt, mark=x, mark options={solid, black}]
  table[row sep=crcr]{%
5 7.3196873380227\\
8 6.23324585377403\\
11 6.00733386622653\\
21 5.85277038124994\\
31 5.82588340145345\\
41 5.81825345322474\\
};

\addlegendentry{exponential}

\addplot [color=black, dashdotted, line width=2.0pt, mark size=2.0pt, mark=triangle, mark options={solid, black}]
  table[row sep=crcr]{%
5 6.57950787090029\\
8 5.9852256504396\\
11 5.91706273017332\\
21 5.83946357025782\\
31 5.82303294626486\\
41 5.81821149584039\\
};
\addlegendentry{pw.\ uniform}

\end{axis}
\end{tikzpicture}%

%% file: figs/quasilin_fvals.tex
%
%
\definecolor{mycolor1}{rgb}{0.00000,0.44700,0.74100}%
\begin{tikzpicture}

\begin{axis}[%
width=2.5in,
height=1in,
at={(0in,0in)},
scale only axis,
xmin=4,
xmax=42,
xtick={ 5,  8, 11, 21, 31,41},
xlabel style={font=\color{white!15!black}},
xlabel={number of time grid points},
ymax=9.3,
ylabel style={font=\color{white!15!black}},
ylabel={closed-loop cost},
axis background/.style={fill=white},
title style={align=center,at={(2.6in,1)}},
legend style={legend pos = north east,legend cell align=left, align=left, draw=white!15!black}
]
\addplot [color=black, dashdotted, line width=2.0pt, mark=*, mark options={solid, black}]
  table[row sep=crcr]{%
5 9.25676136324529\\
8 9.24493943292915\\
11 9.26896752887804\\
21 8.95561290076231\\
31 8.88466092550924\\
41 8.86136203967596\\
};

\addlegendentry{uniform}

\addplot [color=black, dashdotted, line width=2.0pt, mark size=5.0pt, mark=x, mark options={solid, black}]
  table[row sep=crcr]{%
5 9.15331024884586\\
8 8.8990197116562\\
11 8.86394707899693\\
21 8.83227437314232\\
31 8.82638263615178\\
41 8.82389877713854\\
};

\addlegendentry{exponential}

\addplot [color=black, dashdotted, line width=2.0pt, mark size=2.0pt, mark=triangle, mark options={solid, black}]
  table[row sep=crcr]{%
5 8.98001795133487\\
8 8.92745967583251\\
11 8.96498430279112\\
21 8.90856889962106\\
31 8.86646810048441\\
41 8.84531148634032\\
};
\addlegendentry{pw.\ uniform}

\end{axis}
\end{tikzpicture}%